\documentclass[11pt,reqno]{amsart}
\usepackage{amsmath}

\usepackage{amssymb}
\usepackage{amsthm}
\usepackage{enumerate}
\usepackage{graphicx}
\usepackage{tabularx}
\usepackage{caption}
\usepackage{subfigure}
\usepackage{color}
\usepackage{hyperref}
\usepackage{microtype}
\usepackage{todonotes}

\usepackage{chngcntr}
 
\counterwithin{figure}{section}

\usepackage[all]{xy}
\definecolor{BLUE}{rgb}{0.0,0.0,1.0}
\definecolor{RED}{rgb}{1.0,0.0,0.0}
\definecolor{CYAN}{rgb}{0.0,1.0,1.0}

\definecolor{DGREEN}{rgb}{0.0,0.5,0.0}

\definecolor{DORANGE}{rgb}{1.0,0.5,0.0}

\definecolor{purple}{rgb}{0.9,0,0.8}

\newcommand\R{{\mathbb R}}
\newcommand{\reals}{\R}

\newcommand{\abbr}[1]{{\sc\lowercase{#1}}}
\newcommand{\abbrs}[1]{{\sc \footnotesize \lowercase{#1}}}

\theoremstyle{plain}
\newcommand{\E}{\mathbf{E}}
\newtheorem{thm}{Theorem}[section]
\newtheorem{lem}[thm]{Lemma}
\newtheorem{prop}[thm]{Proposition}

\newtheorem{defn}[thm]{Definition}

\newtheorem{Remark-lbl}[thm]{Remark}
\newcommand{\eqnsection}{\renewcommand{\theequation}{\thesection.\arabic{equation}}
      \makeatletter \csname @addtoreset\endcsname{equation}{section}\makeatother}

\eqnsection

\DeclareMathOperator{\var}{Var}
\renewcommand{\P}{\mathbf{P}}

\newcommand{\one}{\mathbf{1}}

\newcommand{\wt}{\widetilde}
\newcommand{\ul}{\underline}
\newcommand{\ol}{\overline}

\newcommand{\CF}{\mathcal {F}}
\newcommand{\CH}{\mathcal {H}}
\newcommand{\CG}{\mathcal {G}}

\newcommand{\CL}{\mathcal {L}}

\newcommand{\CU}{\mathcal {U}}

\newcommand{\CZ}{\mathcal {Z}}

\newcommand{\SA}{\mathsf {A}}
\newcommand{\SB}{\mathsf {B}}
\newcommand{\SC}{\mathsf {C}}

\newcommand{\SH}{\mathsf {H}}
\newcommand{\sh}{\mathsf {h}}
\newcommand{\SG}{\mathsf {G}}

\newcommand{\SJ}{\mathsf {J}}
\newcommand{\SK}{\mathsf {K}}

\newcommand{\SN}{\mathsf {N}}

\newcommand{\SQ}{\mathsf {Q}}

\newcommand{\SW}{\mathsf {W}}

\newcommand{\wh}{\widehat}
\newcommand{\p}{\mathbf{P}}
\newcommand{\Q}{\mathbf{Q}}
\newcommand{\Z}{\mathbb{Z}}
\newcommand{\N}{\mathbb{N}}

\newcommand{\tmix}{t_{\rm mix}}
\newcommand{\thit}{t_{\rm hit}}
\newcommand{\tcov}{t_{\rm cov}}

\newcommand{\tcovp}{t_{\rm cov}^{\square}}
\newcommand{\Green}{G}
\newcommand{\ball}{\SB}
\newcommand{\cyl}{\SC}
\newcommand{\ttorus}{\SG_n(a)}
\newcommand{\loc}{\CL}
\newcommand{\oloc}{\ol{\loc}}
\newcommand{\hloc}{\wh{\loc}}
\newcommand{\nballE}{\SN \ball}
\newcommand{\ballE}{\ol{\nballE}}
\newcommand{\ncylE}{\SN \cyl}
\newcommand{\cylE}{\ol{\SN \cyl}}
\newcommand{\TV}{{\rm TV}}
\newcommand{\giv}{\,|\, }
\newcommand{\Wpq}{W_{\textsc{p-q}}}
\newcommand{\twoD}{2\mathrm{D}}
\newcommand{\threeD}{3\mathrm{D}}

\begin{document}
\title[Cut-off for lamplighter chains on tori: dimension interpolation]{Cut-off for lamplighter chains on tori:\\dimension interpolation and Phase transition}

\author[A.\ Dembo]{Amir Dembo$^\star$}
\author[J.\ Ding]{Jian Ding$^\dagger$}
\author[J.\ Miller]{Jason Miller$^\ddagger$}
\author[Y.\ Peres]{Yuval Peres$^\S$}

\address{$^\star$Department of Mathematics, Stanford University
\newline\indent Building 380, Sloan Hall, Stanford, California 94305}

\address{$^\dagger$Department of Statistics, University of Pennsylvania, Philadelphia, PA 19104
\newline\indent }

\address{$^\ddagger$Statistical Laboratory, University of Cambridge
\newline\indent Centre for Mathematical Sciences, Cambridge CB3 0WB, United Kingdom}

\address{$^{*\dagger\S}$Microsoft Research, Redmond
\newline\indent }

\date{\today}



\thanks{$^{*}$Research partially supported by NSF grants
DMS-1106627 and DMS-1613091.}
\thanks{$^{\dagger}$Research partially supported by NSF grant DMS-1313596, DMS-1757479 and an Alfred
Sloan fellowship.}
\thanks{$^{\ddagger}$Research partially supported by NSF grant DMS-1204894.}
\thanks{Part of the work was done when A.D. and J.D. participated  
in the MSRI program on Random Spatial Processes.}

\begin{abstract}
Given a finite, connected graph $\SG$, the lamplighter chain 
on $\SG$ is the lazy random walk $X^\diamond$ on the associated
lamplighter graph $\SG^\diamond=\Z_2 \wr \SG$. 
The mixing time of the lamplighter chain on the torus $\Z_n^d$ is known to have a cutoff at a time asymptotic to the cover time of $\Z_n^d$ if $d=2$, and to half the cover time if $d \ge 3$.  We show that the mixing time of the lamplighter chain on $\ttorus=\Z_n^2 \times \Z_{a \log n}$ has a cutoff at $\psi(a)$ times the cover time of $\ttorus$ as $n \to \infty$, where $\psi$ is an explicit weakly decreasing map from $(0,\infty)$ onto $[1/2,1)$. In particular, as $a > 0$ varies, the threshold continuously interpolates between the known thresholds for $\Z_n^2$ and $\Z_n^3$. Perhaps surprisingly, we find a phase transition 
(non-smoothness of $\psi$) at the point $a_*=\pi r_3 (1+\sqrt{2})$,
where high dimensional behavior ($\psi(a)=1/2$ for all $a \ge a_*$) 
commences. Here $r_3$ is the effective resistance from $0$ to $\infty$ in $\Z^3$.
\end{abstract}

\maketitle

\section{Introduction}
\label{sec::intro}

\subsection{Setup}
\label{subsec::setup}

Suppose that $\SG$ is a finite, connected graph with vertices $V(\SG)$ and edges $E(\SG)$, respectively. Each vertex $(\ul f,x)$ of the 
\emph{wreath product} $\SG^\diamond = \Z_2 \wr \SG$ 
consists of a $\{0,1\}$-labeling $\ul f$ of $V(\SG)$ 
and $x \in V(\SG)$.
There is an edge between $(\ul f,x)$ and $(\ul g,y)$ 
if and only if $\{x,y\} \in E(\SG)$ and $f_z = g_z$ for all $z \notin \{x,y\}$.  
Recall that the transition kernel of the \emph{lazy random walk} $X$ on $\SG$ is 
\begin{equation}
\label{eqn::lazy_rw_definition}
P(x,y) := \p_x[X_1 = y] = \begin{cases} \frac{1}{2} \quad&\text{if}\quad x = y,\\ \frac{1}{2d(x)} \quad&\text{if}\quad \{x,y\} \in E(\SG), \end{cases}
\end{equation}
where $d(x)$ is the degree of $x \in V(\SG)$ and $\p_x$ denotes the law under which $X_0 = x$. The \emph{lamplighter chain} $X^\diamond$ is the lazy random 
walk on $\SG^\diamond$. Explicitly, it moves from $(\ul f,x)$ by
\begin{enumerate}
\item picking $y$ adjacent to $x$ in $\SG$ according to $P$, then
\item if $y \neq x$, updating each of the values of $f_x$ and $f_y$ independently according to the uniform measure on $\Z_2$ (with $f_z$ unchanged 
for all $z \notin \{x,y\}$).
\end{enumerate}
We refer to $f_x$ as the state of the lamp at $x$. 
If $f_x = 1$ (resp.\ $f_x = 0$) we say that the lamp at $x$ is on (resp.\ off); this is the source of the name ``lamplighter.''  Note that the projection of $X^\diamond$ to $\SG$ evolves as a lazy random walk on $\SG$.  It is easy to see that the unique stationary distribution of $X^\diamond$ is given by the 
product of the (unique) stationary distribution of $P(\cdot,\cdot)$ and the uniform measure over the $\{0,1\}$-labelings of $V(\SG)$.  See Figure~\ref{fig:lamplighter} for an illustration of the lamplighter chain.

The purpose of this work is to determine the asymptotics of the total variation mixing time of the lamplighter chain on a particular sequence of graphs.  In order to state our main results precisely and put them into context, we will first review some basic terminology from the theory of Markov chains.  Suppose that $\mu,\nu$ are measures on a finite probability space.  The \emph{total variation distance} between $\mu,\nu$ is given by
\begin{equation}
\label{eqn::tv_definition}
 \| \mu - \nu \|_{\TV} = \max_{A} |\mu(A) - \nu(A)| = \frac{1}{2} \sum_{x} |\mu(x) - \nu(x)|.
\end{equation}
The $\delta$-\emph{total variation mixing time} of a transition kernel $Q$ on a graph $\SH$ with stationary distribution $\pi(\cdot)$ is given by
\begin{equation}
\label{eqn::tmix_definition}
 \tmix(\SH, \delta) = \min\left\{ t \geq 0: \max_{x \in V(\SH)} \|Q^t(x,\cdot) - \pi(\cdot)\|_{\TV} \leq \delta \right\}.
\end{equation}
Throughout, we let $\tmix(\SH) = \tmix(\SH,\tfrac{1}{2e})$.  Lazy random walk 
$\wh{X}$ on a family of graphs $(\SH_n)$ is said to exhibit \emph{cutoff} if
\begin{equation}
\label{eqn::cutoff}
\lim_{n \to \infty} \frac{\tmix(\SH_n,\delta)}{\tmix(\SH_n,1-\delta)} = 1\quad\quad\text{for all}\quad\quad \delta > 0.
\end{equation}
For each $x \in V(\SH)$ let $\tau_x = \min\{k \geq 0 : \wh{X}_k = x\}$ be the hitting time of $x$.  With $\E_x$ the expectation associated with $\p_x$, the \emph{maximal hitting time} of $\SH$ is given by
\[ \thit = \thit(\SH) = \max_{x,y \in V(\SH)} \E_y[\tau_x]\]
and the \emph{cover time} of $\SH$ is
\[ \tcov =\tcov(\SH) = \max_{y \in V(\SH)} \E_y\left[ \max_{x \in V(\SH)} \tau_x \right].\]

\subsection{Related work}

The mixing time of $\SG^\diamond$ was first studied by H\"aggstr\"om and Jonasson in \cite{HJ97} in the case of the complete graph $\SK_n$ and the one-dimensional cycle $\Z_n$.  Their work implies a total variation cutoff with threshold $\tfrac{1}{2} \tcov(\SK_n)$ in the former case and that there is no cutoff in the latter.  The connection between $\tmix(\SG^\diamond)$ and $\tcov(\SG)$ is explored further in \cite{PR} (see also the account given in \cite[Chapter~19]{LPW}), in addition to developing the relationship between $\thit(\SG)$ and the relaxation time (i.e., inverse spectral gap) of $\SG^\diamond$, and
the relationship between exponential moments of
the size of
the uncovered set $\CU(t)$ of $\SG$ at time $t$ and
the uniform, i.e., $\ell_\infty$-mixing time
of $\SG^\diamond$.
In particular, it is shown in \cite[Theorem~1.3]{PR} that if $(\SG_n)$ is a sequence of graphs with $|V(\SG_n)| \to \infty$ and $\thit(\SG_n) = o(\tcov(\SG_n))$ then
\begin{equation}
\label{eqn::pr_bounds}
 \frac{1}{2}(1+o(1)) \tcov(\SG_n) \leq \tmix(\SG_n^\diamond) \leq (1+o(1))\tcov(\SG_n) \quad\text{as}\quad n \to \infty.
\end{equation}
Related bounds on the order of magnitude of the uniform mixing time and the relaxation with generalized lamps were obtained respectively in \cite{KMP11} and \cite{KP12}.

By combining the results of \cite{DPRZ-late} and \cite{ALD-cover}, it is observed in \cite{PR} that $\tmix((\Z_n^2)^\diamond)$ has a threshold at $\tcov(\Z_n^2)$. Thus,~\eqref{eqn::pr_bounds}
gives the best \emph{universal} bounds,
since $\SK_n$ attains the lower bound and $\Z_n^2$ attains the upper bound.  In \cite{Miller-Peres}, it is shown that $\tmix((\Z_n^d)^\diamond) \sim \tfrac{1}{2} \tcov(\Z_n^d)$ when $d \geq 3$ and more generally that $\tmix(\SG_n^\diamond) \sim \tfrac{1}{2} \tcov(\SG_n)$ whenever $(\SG_n)$ is a sequence of graphs with $|V(\SG_n)| \to \infty$ satisfying certain uniform local transience assumptions.  This prompted the question \cite[Section~7]{Miller-Peres} of whether for each $\gamma \in (\tfrac{1}{2},1)$ there exists a (natural) family of graphs $(\SG_n)$ such that $\tmix(\SG_n^\diamond) \sim \gamma \tcov(\SG_n)$ as $n \to \infty$.  In this work we give an affirmative answer to this question by analyzing the lamplighter chain on a thin $3$D torus.

Cutoff for lazy random walks on $\SG_n^\diamond$ 
is further examined in \cite{DKN-fractal} for a large 
class of fractal graphs $\SG_n$. They show that cutoff 
never occurs for strongly recurrent $\SG_n$ 
(namely of spectral dimension $d_s<2$), while the 
sufficient conditions of \cite{Miller-Peres} for cutoff 
at $\frac{1}{2}\tcov(\SG_n)$, apply for transient $\SG_n$ 
(i.e. having $d_s>2$). However, such universality seem 
to not hold in the setting of $d_s=2$, namely for 
the fractal analog of the $2$D and thin $3$D torus 
considered here.

\begin{figure}[ht!]
\includegraphics[width=0.27\textwidth]{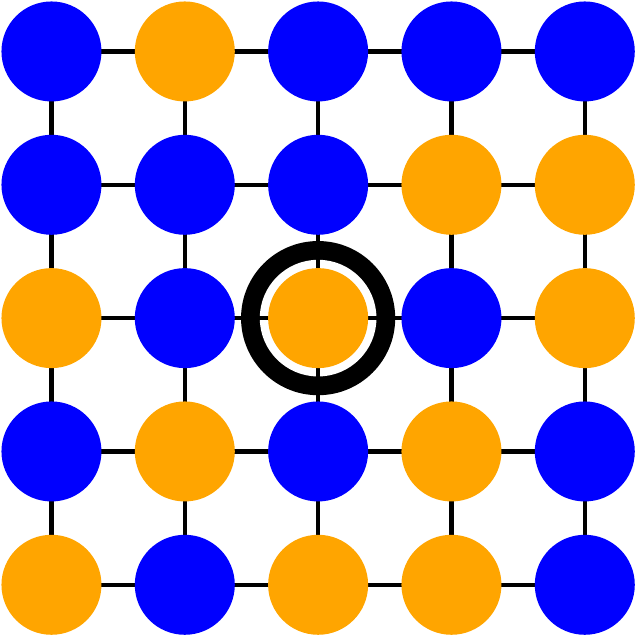}
\caption{\label{fig:lamplighter} Shown is a lamplighter configuration on $\Z_5^2$ (without the wraparound edges).  The state of the lamps is indicated by the colors.  The circle gives the position of the underlying random walker.}
\end{figure}

\subsection{Main results}
\label{subsec::main_results}

Fix $a > 0$.  We consider the mixing time for the \abbr{srw} $X^\diamond_k$,
$k \in \N$, on the
lamplighter graph $(\ttorus)^\diamond$ for the $3$D thin tori $\ttorus= (V_n,E_n) = \Z_n^2 \times \Z_h$
of size $n \times n\times h$, where $h = [a\log n]$.
From the main result of \cite{DPRZ-cover} we know
that the cover time of the $2$D projection of
\abbr{srw} on $\ttorus$ to $\Z_n^2$ is given by
\[
\tcovp :=\frac{3}{2} \tcov( \Z_n^2)  \quad\quad\text{where}\quad\quad
\tcov(\Z_n^2) := \frac{4}{\pi} n^2 (\log n)^2(1+o(1))
\]
(where the factor $\tfrac{3}{2}$ is due to the lazy steps
of walk in the $h$-direction, which occur with probability~$\tfrac{1}{3}$).   Let
\begin{equation}
\label{eqn::phi_def}
\phi := \pi r_3 a
\end{equation} where $r_3$ denotes the resistance $0 \leftrightarrow \infty$ for the \abbr{srw} in $\Z^3$. That is,
\begin{equation}
\label{eqn::r3_def}
r_3= \frac{1}{6 q} \quad \text{where}  \quad q = \p_0[T_0 = \infty] ,
\end{equation}
and $T_0$ denotes the return time to zero by \abbr{srw} in $\Z^3$
(see \cite[Proposition~9.5]{LPW}
for the relation~\eqref{eqn::r3_def} and an explicit formula for $q$).
In Section~\ref{sec:cover}, we use the recent development
which relates cover time with the extremes of Gaussian fields, see \cite{Ding11}, to establish the following theorem.
\begin{thm}
\label{thm-cover-time}
The cover time $\tcov(a, n)$ of $\ttorus$ by \abbr{srw} is given by
\[ \tcov(a, n) = (1+o(1)) C(a,n), \quad \mbox{ as } n\to \infty\]
where
\begin{equation}\label{eq:cover-time}
C(a,n) := (1 + 2 \phi) \tcovp
\end{equation}
and $\phi$ is as in~\eqref{eqn::phi_def}.
\end{thm}

\begin{Remark-lbl} One expects the 
cutoff threshold transition from $2$D to $3$D behavior
to occur when $\tcov(\ttorus)/\tcov(\Z_n^2)=O(1)$, while 
depending on the height multiplier $a$. By Theorem \ref{thm-cover-time} 
the correct scaling for this is $\log n$ (which as shown in 
Section~\ref{sec:cover}, has to do with the decay rate of 
${\mathrm{Diam}}_{R_{\mathrm{eff}}}(\Z_n^2)$, see \eqref{eq:bd-reff}).
\end{Remark-lbl}

\begin{figure}[ht!]
\includegraphics[width=0.7\textwidth]{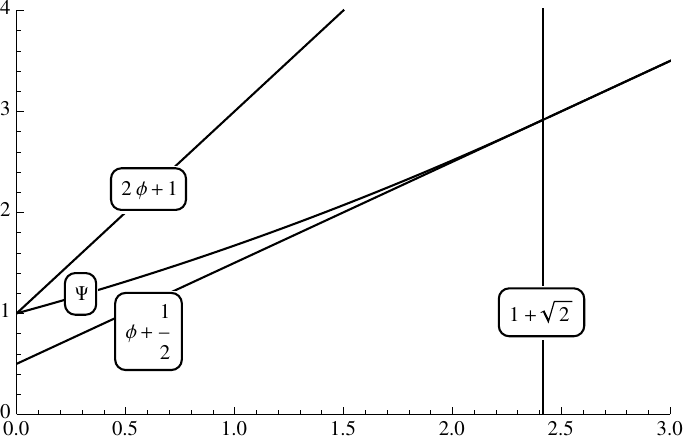}
\caption{\label{fig:mixing_plot} \small{
The function $\Psi$ from~\eqref{eq:mix-bd} which
gives the asymptotic ratio of $\tmix/\tcovp$.  
Also shown are the
bounds of $2\phi+1$ and $\phi+\tfrac{1}{2}$
on $\tmix / \tcovp$; recall~\eqref{eqn::pr_bounds}.  
The lower bound is attained by $\Psi$ starting at 
$\phi=1+\sqrt{2}$. 
}}
\end{figure}

Our main result establishes cutoff for 
\abbr{srw} $\{X^\diamond_k\}$ on the lamplighter graph $(\ttorus)^\diamond$ and determines its location as a function of the height parameter $a$.
\begin{thm}
\label{thm:cut-off}
Total-variation cut-off occurs for
$\{X^\diamond_k\}$ on $\ttorus$
at $\Psi(\phi) \tcovp$, where
\begin{equation}
\label{eq:mix-bd}
\Psi(\phi) :=
\begin{cases}
 \left(1+ (1-\frac{1}{\sqrt{2}}) \phi \right)^2, &\mbox{ if } \phi \leq \sqrt{2} + 1, \\
\frac{1+2\phi}{2}, &\mbox{ if } \phi > \sqrt{2} + 1\,.
\end{cases}
\end{equation}
In particular, $\tmix = (\Psi(\phi) + o(1)) \tcovp$.
\end{thm}

Comparing Theorems~\ref{thm-cover-time}
and~\ref{thm:cut-off} we see that the ratio between
the mixing time of $\{X_k^\diamond\}$ and the
cover time $C(a,n)$ of the
base graph by the \abbr{srw} $\{X_k\}$,
monotonically interpolates between the fraction of the cover time necessary to mix in two dimensions (ratio $1$) \cite{DPRZ-cover,PR} and the fraction in three dimensions (ratio $1/2$) \cite{Miller-Peres}.  This gives an affirmative answer to the first question posed in \cite[Section~7]{Miller-Peres}.  See Figure~\ref{fig:mixing_plot} for a plot of the quantities from Theorem~\ref{thm:cut-off} and how they relate to the bounds~\eqref{eqn::pr_bounds}.

We note in passing that for all $\phi>0$ the value of
$\tmix/\tcovp \to \Psi(\phi)$ is bounded
away from its trivial bound $1$. The latter corresponds
to the mixing time for the lamplighter graph
on the $2$D torus of side length $n$ that corresponds to the base sub-graph 
$(x_1,x_2,1)$ of $\ttorus$ (which as shown in \cite{PR} coincides with
the cover time $\tcovp (1+o(1))$ for the corresponding (lazy)
$2$D projected \abbr{srw}). However, when $\phi \ge \sqrt{2}+1$
asymptotically $\tmix$ matches the
elementary bound $\tmix \ge \frac{(1+o(1))}{2} C(a, n)$
(see~\eqref{eq:cover-time}, and \cite[Lemmas~19.3 and~19.4]{LPW}), which 
applies for the lamplighter chain on any base graph having
maximal hitting time which is significantly smaller
than the corresponding cover time.

\begin{Remark-lbl} It is possible to adapt the proof of 
Theorem~\ref{thm:cut-off} so that it will yield a similar conclusion in the 
setting of a more general $3$-dimensional lattice confined to 
a thin slab of size $n \times n \times h$.

\end{Remark-lbl}
\begin{Remark-lbl}\label{rmk:unvisited}
Clearly, $X^\diamond_t$ is not mixed for as long as the uncovered set 
$\CU(t)$ of $X$ exhibits some non-trivial systematic 
geometric structure that makes the corresponding lamp states
distinguishable from the uniform marking of $V(\SG)$ by i.i.d.\ fair coin flips.
Further, the uniformity of $\CU(t)$ typically determines the
threshold $t$ for mixing time of $X^\diamond$, and indeed our work contributes to the literature on the 
geometric structure of the last visited points by the \abbr{srw} 
(see \cite{BH91,DPRZ-cover,DPRZ-late,Miller-Peres,BEL12,Miller-Sousi}).  
\end{Remark-lbl}

\begin{Remark-lbl} By the reasoning of Remark \ref{rmk:unvisited}, 
up to technical issues, we expect that $\tmix(\SG_n^\diamond)$ is 
$\gamma \tcov(\SG_n)(1+o(1))$ for some $\gamma \in (1/2,1)$, provided that: 
\begin{itemize}
\item The Green's functions $G_n(x,y)$ for $\SG_n$ are bounded above 
on the diagonal.  (This should prevent clustering in $\CU(\gamma \tcov(\SG_n))$ for 
$\gamma$ sufficiently close to $1$.)
\item The decay of $G_n(x,y)$ in terms of the distance between $x$ and $y$ is
non-uniform in $n$. 
(This should lead to clustering in $\CU(\gamma \tcov(\SG_n))$ beyond $\gamma=\tfrac{1}{2}$, while \cite{Miller-Peres} show that a uniform decay rate results in the threshold at $\tfrac{1}{2}\tcov(\SG_n)$.)
\end{itemize}
One interesting family of graphs $\SG_n$ of this type is given by the infinite cluster for super-critical Bernoulli percolation restricted to a thin
slab of size $n \times n \times h$.
\end{Remark-lbl}

\subsection{Outline of the proof of Theorem \ref{thm:cut-off}}\label{subsec:pf-outline}

Fixing $s \ge 1$, 
for any $\rho,z \in [0,1]$, the functions 
\begin{equation}\label{eq:b-def}
b_\rho (z)=1-\rho-\frac{s(1-z)^2}{1-\rho} , 
\qquad \qquad \alpha_\rho (z) = \frac{s z^2}{\frac{\rho}{2} + \phi} ,
\end{equation}
control the structure of $\CU(s \tcovp)$. 
Specifically, for any
$\rho \in [0,1]$ we associate with each $x \in V_n$
a type $z \in [0,1]$ according to the number of excursions
of the \abbr{srw}, by time $s \tcovp$, across the
$2$D cylindrical annulus of radii $M h n^\rho$ and
$M^2 h n^{\rho}$, centered at the $2$D projection of $x$.
Our parameters are such that for $n \to \infty$
followed by $M \to \infty$, \abbrs{whp} 
about $n^{2 b_\rho(z) + o(1)}$ of the $n^{2 (1-\rho) + o(1)}$
such annuli are of $z$-type and points $x \in V_n$
whose 2D projection is not far from the center of 
such $z$-type annulus, are unvisited
by the \abbr{srw} with probability $n^{-\alpha_\rho(z)+o(1)}$.
Further, in Section \ref{sec:ubd-gen-str} we confirm 
the following representation of $\Psi(\phi)$. 
\begin{lem}\label{lem:Psi}
For $s \ge 1$ and $\rho,z \in [0,1]$ let 
$b_\rho(z)$, $\alpha_\rho(z)$ be as in \eqref{eq:b-def}, with the convention that $b_1(z)=-\infty {\bf 1}_{\{z \ne 1\}}$.
Then, $\Psi(\cdot)$ of~\eqref{eq:mix-bd} emerges
from the following variational problem:
\begin{align}\label{eq:var-pbm}
\Psi(\phi) &=\inf\{ s \ge 1 : \forall \rho,z \in [0,1], \;
b_\rho(z) \ge 0 \;\;\Longrightarrow \;\;\alpha_\rho(z) \ge \rho \}
\\
&= \sup\{ s \ge 1 : \exists \rho,z \in [0,1], \mbox{ such that } b_\rho(z) \ge 0 \; \mbox{ and } \; \alpha_\rho(z) \le \rho \}.
\label{eq:t-lbd}
\end{align}
\end{lem}
\noindent
Calling a $z$-type $\rho$-\emph{admissible} if 
$b_\rho(z) > 0$, we know from \eqref{eq:t-lbd} 
that for any $s < \Psi(\phi)$ there exist $\rho \in (0,1)$ 
and $\rho$-admissible $z' \in (0,1)$ with $\alpha_\rho(z') < \rho$.
By continuity, the same applies for $L$ large enough and 
$\rho_k = k/L$ with $k:=[\rho L]$. Using this approximation, 
we show in Section~\ref{sec-lbd} that the maximum discrepancy at time
$s \tcovp$ between ``off-lamps'' and ``on-lamps''
over a certain large enough (and spatially well separated)
collection $\SA_{\twoD,k}$ of $2$D disjoint cylinders
of radii $h n^{\rho_k}$, far exceeds its value under the
invariant (uniform) law for the \abbr{srw} $\{X_\cdot^\diamond\}$.
This statistics distinguishes between the law of 
the lamplighter chain at time $s \tcovp$ and
its stationary law, thereby yielding the stated
lower bound on $\tmix=\tmix(\ttorus^\diamond)$.
 
In contrast, by the dual variational problem \eqref{eq:var-pbm},
for $s>\Psi(\phi)$, if $b_\rho(z) \ge 0$ then
the discrepancy of about $n^{-\alpha_\rho(z)}$ between the
fractions of ``off-lamps'' and ``on-lamps''
within each such annulus,
is buried under the inherent noise level of $n^{-\rho}$.
Thus, all such statistics agree with the stated upper bound 
$\tmix \le \Psi(\phi) \tcovp$.
As explained in Section \ref{sec:ubd-mix}, to actually 
upper bound $\tmix$, one needs to  
control exponential moments of the size of $\CU(s\tcovp)$
(more precisely, the size of the intersection of the unvisited 
sites by two independent random walks), 
which is the main technical challenge here. This is carried
out by carefully estimating the number of excursions within consecutive 
annuli. Specifically, utilizing H\"older's inequality it 
suffices to separately consider each $z$-type and to 
do so on a certain sparse sub-lattice $\SA$ 
of $V_n$, where at 
$\rho=0$ the Bernoulli$(n^{-\alpha_\rho(z)})$ variables
corresponding to $z$-type unvisited sites in $\SA$
are approximately independent
even in terms of tail probabilities.

At any $\rho>0$
the corresponding Bernoulli($n^{-\alpha_\rho(z)}$)
variables are no longer asymptotically independent.
To circumvent this problem, we group the vertices of 
$\SA$ into nested, growing cylindrical annuli, 
centered at sub-lattices $\SA_{\twoD,k}$ that correspond 
to $\rho_k=k/L$, $k=0,1,\ldots,L$. Then, for each 
vertex/base point, the excursion counts across different 
scale annuli define a type profile 
$\underline{z} \in [0,1]^{L+1}$ (that coincide at $k=0$ 
with its $z_0$-type). We characterize the collection
of all \emph{possible} excursion count profiles by
a careful extension of the concept of 
$\rho$-admissible $z$-types to that of admissible
$\underline{z}$-types. The bulk of this article is 
thus about controlling the exponential moment of the 
number of unvisited sites per
fixed admissible $\underline{z}$-type.
Taking first $n \to \infty$, then $M \to \infty$ and
finally $L \to \infty$, this is done in 
Sections~\ref{sec:ubd-mix}--\ref{sec:2d}
via estimates on modified Green functions and utilizing stochastic domination to employ 
large deviation tail estimates for sums 
of i.i.d. variables.

We note in passing that while lower bounding 
$\tmix$ we find that the most likely way to have $z$-type
at the $O(h)$ size 2D annulus corresponding to $\rho=0$,
is via the profile $z(\rho)=1-(1-\rho)(1-z)$.  However, we also show in Section~\ref{sec-lbd} that 
such profiles are \emph{highly unlikely} for the set $\CU(s \tcovp)$.
Thus, for a sharp upper bound on $\tmix$ one must
control the large deviations of 
\emph{all admissible} $z(\cdot)$-type profiles.

\section{Cover time for the thin torus: proof of Theorem~\ref{thm-cover-time}}
\label{sec:cover}

The Gaussian Free Field (in short \abbr{GFF}), on finite, connected graph $\SG = (V, E)$,
with respect to some fixed $v_0 \in V$, is the stochastic process
$\{\eta_u\}_{u \in V}$ with $\eta_{v_0}=0$, whose
density with respect to Lebesgue measure on
$V \setminus \{v_0\}$ is proportional to
\begin{equation}
\label{eq:density}
\exp\Big(-\frac14 \sum_{u\sim v} |\eta_u-\eta_v|^2
\Big),
\end{equation}
where we used $u \sim v$ to denote
 $\{u,v\} \in E$.
An important connection between \abbr{GFF} and the
\abbr{srw} on $\SG$ is the following identity
(see for example, \cite[Theorem~9.17]{Janson97}):
\begin{equation}\label{eq-GFF-resistance}
\E[(\eta_u - \eta_v)^2] = R_{\mathrm{eff}}(u, v)\,.
\end{equation}
Here $R_{\mathrm{eff}}(u, v)$ is the effective resistance between $u , v\in V$ in
the electrical network associated with $\SG$
by placing a unit resistor on each edge $\{u,v\} \in E$ (and we sometimes use $R^{\SG}_{\mathrm{eff}}(u, v)$ to emphasize the
underlying graph $\SG$, in case of possible ambiguity).

Our proof of Theorem~\ref{thm-cover-time} relies on the
following relation between the cover time
$\tcov(\SG)$ of $\SG$ by \abbr{srw} and the
maximum of the corresponding \abbr{GFF}.

\begin{thm}[$\text{\cite[Theorem~1.1]{Ding11}}$]
\label{thm-Ding}
Consider a sequence of graphs $\SG_n = (V_n, E_n)$ of
uniformly bounded maximal degrees, such that $\thit(\SG_n) =
o(\tcov(\SG_n))$ as  $n \to \infty$. For each $n$, let
$\{\eta_v\}_{v\in V_n}$ denote a \abbr{GFF} on $\SG_n$
with $\eta_{v_0^n} = 0$ for certain $v_0^n \in V_n$.
Then, as $n  \to \infty$,
\begin{equation}
\label{eq-cover-asymp}
\tcov(\SG_n) =
(1+o(1))|E_n| \Big(\E \, [\{ \sup_{v\in V_n} \eta_v \} ] \Big)^2\,.
\end{equation}
\end{thm}
In light of the preceding theorem, the key to the proof of Theorem~\ref{thm-cover-time} is an estimate on the expected supremum for the associated \abbr{GFF}. To this end, we start with few estimates of effective resistances
assuming familiarity with the connection
between random walks and electric flows (see for example \cite[Chapter~2]{LP}).
\begin{lem}\label{lem-send-flow}
Let $\{X_n\}$ denote the \abbr{srw} on the graph
$\SG = (V, E)$ started at some $o\in V$, independent
of a Geometric random variable $T$.  Then, there exists a current flow
$\theta =\{ \theta_{u,v} : \{u,v\} \in E\}$
with unit current source at $o$, current $p_v := \P[X_T = v]$ reaching
each $v \in V$, and the
Dirichlet energy bound
$$
{\mathcal D}(\theta) :=
\sum_{(u,v) \in E} \theta_{u,v}^2 \leq
\tfrac{1}{d_o}
\E \Big[\sum_{n=0}^T \one_{\{X_n = o\}} \Big]\,.
$$
\end{lem}
\begin{proof}
Let $t=\P[T \ge 1] \in (0,1)$. Set $\oloc(v) := \tfrac{1}{d_v} \E [\sum_{n=0}^T \one_{\{X_n = v\}}]$ and
$N(u,v):=\sum_{n=0}^{T-1} \one_{\{X_{n}= u, X_{n+1} = v
\}}$, for each $u,v\in V$.
Then, due to the memory-less
property of Geometric random variables, clearly
$$
p_v = \one_{v=o} +
\sum_{u:u\sim v} ( \E [ N(u, v) ] - \E [ N(v,u)] ) = \one_{v=o} + t \sum_{u:u\sim v}(\oloc(u) - \oloc(v)) \,.
$$
Thus, the current flow $\theta^\star_{u,v} := t (\oloc(u) - \oloc(v))$ on $(u,v) \in E$, together with
external unit current into $o$, results with
current $p_v$ reaching each $v \in V$. Furthermore,
\begin{align*}
\sum_{ (u , v) \in E} (\theta^\star_{u, v})^2 &= \tfrac{t^2}{2}
\sum_{ (u , v) \in E} (\oloc(u) - \oloc(v))^2 \leq
t \sum_{u\in V} ( \oloc(u) \sum_{v: v\sim u}
(\oloc(u) - \oloc(v)) ) \\
& \leq  t \oloc(o) \sum_{v:v\sim o} (\oloc(o) - \oloc(v)) \leq \oloc(o),
\end{align*}
since $t \sum_{v: v\sim u} (\oloc(u) - \oloc(v)) = - p_u \leq 0$ for all $u \neq o$, and is at most one at $u=o$.
\end{proof}

We will also need the following claim.
\begin{lem}\label{lem-bad-jian}
For any graph $\SG=(V, E)$, let $R$ be the diameter for the effective resistance
(of the \abbr{srw}, namely with unit edge weights).
Consider a collection of numbers $\{\rho_v: v\in V\}$ such that $\sum_{v\in V} \rho_v = 0$ and $\frac{1}{2}\sum_{v\in V} |\rho_v| = 1$, and let $\Theta$ denote the collection of all flows on $\SG$
such that at any vertex $v$ the difference between
out-going and in-coming flow is $\rho_v$. Then,
$$
\min_{\theta\in \Theta} \{\mathcal D(\theta) \} \leq R\,.
$$
\end{lem}
\begin{proof}
Let $V^+ = \{v\in V: \rho_v \geq 0\}$ and $V^- = V\setminus V^+$. We define a function $w: V^+ \times V^- \mapsto [0, \infty)$ by $w(v, u) = |\rho_v \rho_u|$.  By assumption on $\rho$, we see that
$$\sum_{u\in V^-} w(v, u) = \rho_v \mbox{ for all } v\in V^+ \mbox{ and } \sum_{u\in V^+} w(u, v) = -\rho_v \mbox{ for all } v\in V^-\,.$$
So in particular we have $\sum_{v\in V^+, u\in V^-} w(v, u) = 1$.
For $(v, u)\in V^+\times V^-$, let $\theta^{v, u}$ be an electric current which sends {unit}
amount of flow from $v$ to $u$ {(so in
particular ${\mathcal D}(\theta^{v,u}) \le
R_{\mathrm {eff}}(v,u)$)}. Denoting
$\theta := \sum_{v\in V^+, u\in V^-}
{w(v, u)}
\theta^{v, u}$, {by our construction of $w(\cdot,\cdot)$} we see that $\theta\in \Theta$. It remains to bound the Dirichlet energy of $\theta$.  By Cauchy-Schwarz inequality, we get that
\begin{align*}
\mathcal D(\theta) = \sum_{e\in E} \theta_e^2 = \sum_{e\in E} \Big(\sum_{v\in V^+, u\in V^-} {w(v,u)} \theta^{v, u}_e\Big)^2 &\leq \sum_{e\in E} \sum_{v\in V^+, u\in V^-}w(v, u) (\theta^{v, u}_e)^2 \\
&\leq \sum_{v\in V^+, u\in V^-} w(v, u)\mathcal D(\theta^{v, u}) \leq R,
\end{align*}
completing the proof of the lemma.
\end{proof}

\begin{lem}\label{lem-green-function-2}
With $R_{\mathrm{eff}}(\cdot, \cdot)$ denoting effective resistances on $\ttorus=(V_n,E_n)$, we have
that for all $x, x'\in V_n$,
\begin{equation}
\label{eq:jian-ubd}
R_{\mathrm{eff}}(x, x') \leq 2 r_3 + \tfrac{1}{a \pi} + o(1)\,.
\end{equation}
Furthermore, for $x=(y, 0)$ and $x'= (y', 0)$ where
$y, y' \in \Z^2$
and $\|y - y'\|_{\Z^2_n} \geq 2a \log n$, we have
\begin{equation}
\label{eq:jian-lbd}
R_{\mathrm{eff}}(x,x') = 2 r_3 + \tfrac{1}{\pi a\log n}(\log \|y - y'\|_{\Z_n^2}) + o(1) \,.
\end{equation}
\end{lem}
\begin{proof} Fixing arbitrary $x,x' \in V_n$ we
establish~\eqref{eq:jian-ubd} upon constructing
a flow of $1+o(1)$ current from $x$ to $x'$ whose
Dirichlet energy is at most $2 r_3 + 1/(a \pi) + o(1)$.
To this end, for $\{X_n\}$ a \abbr{srw} on $\ttorus$
and an independent Geometric random variable $T$ of mean $(\log n)^4$,
let
$p_v = \P_x[X_T = v]$ for $v\in V_n$, and
$p_{[i]} := \sum_{v\in \Z_n^2 \times \{i\}} p_v$
(namely, the probability that the ``vertical'' coordinate
of $X_T$ is at $i \in \Z_h$).
We claim that
\begin{equation}\label{eq:exp-loc-time}
\frac{1}{6}\E_x \Big[ \sum_{t=0}^T \one_{\{X_t = x\}} \Big] = r_3 +o(1)\,.
\end{equation}
In order to see the lower bound in~\eqref{eq:exp-loc-time}, we note that the random walk is the same as a random walk in $\Z^3$ in the first {$h=[a \log n]$} steps, during which period the expected number of visits accumulated at $x$ is {already} $6(r_3+o(1))$.
Setting {$N=(\log n)^5$}, since
$\E (T \one_{T \ge N}) \to 0$, we get the
matching upper bound upon showing that
\begin{align}
\E_x\Big[\sum_{t=h}^{N} \one_{\{X_t = x\}}\Big ]
=o(1)
 \,. \label{eq-bad-jian-1}
\end{align}
To this end, with $A$ denoting the event that
simultaneously for all $h \leq t \leq N$,
the number of vertical steps made by the
\abbr{srw} up to time $t$ is in the range $(t/10, t/2)$,
we clearly have that {$\P[A^c] \leq (\log n)^{-r}$
for any $r$ finite and all $n$ large enough}. Therefore
\begin{align*}
\E_x\Big[\sum_{t=h}^{N} \one_{\{X_t = x\}}  \Big ]& \leq N \P[A^c] + \E_x\Big[\sum_{t=h}^{N}
\one_{\{X_t = x, \,A\}}  \Big ]
= o(1)+\sum_{t=h}^{N} \frac{O(1)}{\sqrt{\log n}} \frac{O(1)}{t} = o(1),
\end{align*}
with the term $\frac{O(1)}{\sqrt{\log n}}$ upper bounding
the probability of the \abbr{srw} returning at time $t$
to its starting height (referring to its vertical coordinate), and $O(1/t)$ bounding the probability of
its 2D projection returning to the starting point, respectively (we obtain their independence
upon conditioning on the number of
vertical steps the \abbr{srw} made up to time $t$).
Combined with~\eqref{eq-bad-jian-1}, this completes the verification of~\eqref{eq:exp-loc-time}.

Now, by~\eqref{eq:exp-loc-time} and Lemma~\ref{lem-send-flow}, there exists a
unit current flow $\theta^{(x)}$ out of $x$, with
current inflow of $p_v$ into each $v\in V_n$ and
\begin{equation}\label{eq-flow-1}
{\mathcal D}(\theta^{(x)})
= \sum_{ (u,v) \in E_n }
(\theta^{(x)}_{u,v})^2 \leq r_3 +o(1)\,.
\end{equation}
Setting
$p'_v := \P_{x'} [X_T = v]$
and $p'_{[i]} := \sum_{v\in \Z_n^2 \times \{i\}} p'_v$, we have by the same reasoning a unit current
flow $\theta^{(x')}$ out of $x'$, with current inflow
$p'_v$ into each $v \in V_n$ and
\begin{equation}\label{eq-flow-2}
{\mathcal D}(\theta^{(x')})
\leq r_3 +o(1)\,.
\end{equation}
Furthermore,
it is clear that  with probability $1-o(h^{-4/3})$ we have $T \geq h^{5/2}$, and thus by time $T$ the vertical
component of $\{X_{t}\}$ is so nearly uniformly
distributed that  (here we use the fact that the mixing time for a cycle of size $k$ is $O(k^2)$ and we apply this fact to the random walks started at $x$ and $x'$ separately)
\begin{equation}\label{eq-uniform}
\max_{i} \big|h p_{[i]} - 1 \big| = o(1)
= \max_i \big|h p'_{[i]} - 1 \big|\,.
\end{equation}
Next, fixing $i \in \Z_h$ set
$\rho_i, \rho'_i\in [0,1]$ such that

$$
\rho_i p_{[i]}
= \rho'_i p'_{[i]} = \min\{p_{[i]},p'_{[i]}\}
$$
so there exist zero-net current flows
on the sub-graph $\Z_n^2 \times \{i\}$
of $\ttorus$, with outflow $\rho_i p_v$ and
inflow $\rho'_i p'_v$ at each
$v\in \Z_n^2 \times \{i\}$. Let
$\theta^i$ denote the flow of minimal
Dirichlet energy among all such current flows and
$|\theta^i| = \frac{1}{2}\sum_{v\in \Z_n^2 \times \{i\}} |\rho_i p_v - \rho'_i p'_v|$ {its total flow}.
Then, by Lemma~\ref{lem-bad-jian} we have that
$$\mathcal D(\theta^i) \leq |\theta^i|^2 \mathrm{Diam}_{R_{\mathrm{eff}}}(\Z_n^2),$$
where $\mathrm{Diam}_{R_{\mathrm{eff}}}(\Z_n^2)$ is the diameter for the resistance metric in the torus $\Z_n^2$.
Note that 
$$\sum_i (\theta^i)^2 \leq \max_i |\theta^i| \sum_i |\theta^i| \leq \max_i |\theta^i|\,,$$
and that thanks to \eqref{eq-uniform},
\begin{equation}\label{eq:bd-theta}
|\theta^i| \leq \frac{1}{2} \sum_{v\in \mathbb Z_n^2 \times \{i\} } |\rho_i| p_v + |\rho'_i| p'_v = \min\{p_{[i]}, p'_{[i]}\} \,
\le \frac{1+o(1)}{h} \,.
\end{equation}
Combining the three preceding inequalities we obtain that
$$\sum_i {\mathcal D}(\theta^i) \leq \frac{1+o(1)}{h}
\mathrm{Diam}_{R_{\mathrm{eff}}}(\Z_n^2)\,. $$
Combined with the standard estimate
\begin{equation}\label{eq:bd-reff}
\mathrm{Diam}_{R_{\mathrm{eff}}}(\Z_n^2)
\le \frac{1+o(1)}{\pi} \log n
\end{equation}
(see, e..g, \cite[Lemma~3.4]{Ding11b}), we arrive at
\begin{equation}
\label{eq:jian-sum-theta-i}
\sum_i {\mathcal D}(\theta^i) \leq \frac{1+o(1)}{h}
\mathrm{Diam}_{R_{\mathrm{eff}}}(\Z_n^2)
\leq \frac{1}{a\pi} (1+o(1)) \,.
\end{equation}
Consider now the current flow $\theta^\star$
from $x$ to $x'$ obtained by combining
$\theta^{(x)}$ with the union of all flows
$\{\theta^i, i \in {\Z}_h\}$ and
the current flow $-\theta^{(x')}$.
The net amount of current reaching sub-graph
$\Z_n^2 \times \{i\}$ is then
$p_{[i]}-p'_{[i]}$, so by~\eqref{eq-uniform}
the flow from $x$ to $x'$ via $\theta^\star$ is
1+o(1), whereas
by~\eqref{eq-flow-1},~\eqref{eq-flow-2} and~\eqref{eq:jian-sum-theta-i}, its Dirichlet energy
is at most

$$
{\mathcal D}(\theta^{(x)}) + \sum_i {\mathcal D}(\theta^i)
+ {\mathcal D}(\theta^{(x')})
\leq  2 r_3 + \tfrac{1}{a \pi} + o(1)
,
$$
completing the proof of the upper bound~\eqref{eq:jian-ubd}.

For the lower bound, we let $\SQ_x$ and $\SQ_{x'}$ be cubes of side-length
$\log\log n$ centered around $x$ and $x'$, respectively.
Let $\SG_{a,n}$ be the graph obtained
by identifying $\partial \SQ_x$  (also $\partial \SQ_{x'}$) as a single vertex,
as well as identifying
$\{(z,i): 1\leq i\leq h \}$ as a single vertex for each
$z\in \Z_n^2$. By Rayleigh monotonicity principle, we see that
$$
R_{\mathrm{eff}}(x, x') \geq R_{\mathrm{eff}}
(x, \partial \SQ_x) +
R_{\mathrm{eff}}(x', \partial \SQ_{x'}) +
R_{\mathrm{eff}}^{\SG_{a, n}}(\partial \SQ_x, \partial \SQ_{x'})\,.$$
It is clear that $R_{\mathrm{eff}}(x, \partial \SQ_x) = R_{\mathrm{eff}}(x', \partial \SQ_{x'}) =r_3+o(1)$.
In addition, by the triangle inequality we see that
\begin{align*}
R_{\mathrm{eff}}^{\SG_{a, n}}(\partial \SQ_x, \partial \SQ_{x'}) & \geq R_{\mathrm{eff}}^{\SG_{a, n}} (x, x') - R_{\mathrm{eff}}^{\SG_{a, n}} (x,  \partial \SQ_x) - R_{\mathrm{eff}}^{\SG_{a, n}} (x',  \partial \SQ_{x'})\\
&= \frac{1}{h} (R_{\mathrm{eff}}^{\Z_n^2}(y, y') -
2 R_{\mathrm{eff}}^{\Z_n^2}(o, \partial \wt{\SQ}_o)) \\
&= \tfrac{1}{\pi a \log n}(\log \|y - y'\|_{\Z_n^2}) + o(1),
\end{align*}
where $\wt{\SQ}_o$ is a 2D box of side-length
$\log\log n$ centered around $o$, and the last equality follows for example from \cite[Lemma~3.4]{Ding11b}. Altogether, this gives the desired lower bound on the effective resistance.
\end{proof}

The following lemma is useful in comparing the maxima of two Gaussian processes (see for example \cite[Corollary~2.1.3]{Fernique75}).
\begin{lem}[Sudakov-Fernique]\label{lem-sudakov-fernique}
Let $\SJ$ be an arbitrary finite index set and let $\{\eta_j\}_{j\in \SJ}$ 
and $\{\xi_j\}_{j\in \SJ}$ be two centered Gaussian processes such that
\begin{equation}\label{eq-compare-assumption}
\E(\eta_j - \eta_k)^2 \geq \E (\xi_j - \xi_k)^2, \mbox{ for all } j, k \in \SJ \,.
\end{equation}
Then $\E [\max_{j\in \SJ} \eta_j] \geq \E [\max_{j\in \SJ} \xi_j]$.
\end{lem}
We are now ready to estimate the maximum of the \abbr{GFF} on the thin torus.
\begin{lem}\label{lem-GFF-torus}
Let $\{\eta_v: v\in V_n\}$ be a \abbr{GFF} on $\ttorus$ with $\eta_{v_0} = 0$. Then,
$$
\E \big[\max_{v \in V_n} \eta_v\big] = 2\sqrt{ r_3 + \tfrac{1}{2a \pi}+o(1)}\sqrt{ \log n} \,.$$
\end{lem}
\begin{proof}
We first prove the upper bound. By~\eqref{eq-GFF-resistance} and Lemma~\ref{lem-green-function-2}, we get that 
\[
\sup_{u,v \in V_n} \{ \var(\eta_u - \eta_v) \} := 2 \sigma_n^2 
\leq 2 r_3 + \tfrac{1}{a \pi} + o(1) \,.
\]
Thus, for i.i.d.\ centered Gaussian variables $\{X_u: u\in V_n\}$ 
of variance $\sigma_n^2$ we have by Lemma~\ref{lem-sudakov-fernique} that
\begin{equation}\label{eq-gaussian-upper}
\E [\max_{u\in V_n} \eta_u] \leq \E [\max_{u\in V_n} X_u]\,.
\end{equation}
Note that
\begin{equation}\label{eq:x-tail}
\E [\max_{u \in V_n} X_u ]  \leq \int_0^\infty \Big[ \big(\sum_{u \in V_n} \P(X_u  \geq r) \big)\wedge 1\Big] \,dr\,.
\end{equation} 
Further, for a centered  Gaussian variable $Y$ of variance $\sigma^2$ 
we have 

\[
\P(Y\geq r) \leq  \mathrm{e}^{-\frac{r^2}{2\sigma^2}} \,, \qquad \forall r \ge 0 \,.
\]
Combined with \eqref{eq:x-tail} it yields that 
$\E [\max_{u \in V_n} X_u] \leq 2\sigma_n \sqrt{ \log n} (1+o(1))$, so 
from \eqref{eq-gaussian-upper} and the bound on $\sigma_n$ we deduce  
the stated upper bound on $\E [\max_{u \in V_n} \eta_u]$.

 For the lower bound, we employ a comparison argument. Let $\SA$ be a 2D box of side-length $n/(8h)$, and let
$\{\xi_v: v\in \SA\}$ be a \abbr{GFF} on $\SA$ with Dirichlet boundary condition (i.e., $\xi|_{\partial \SA} = 0$). Now define mapping $g: \SA \mapsto \ttorus$ by $g(v) = (2 h v, 0)$. It is well known that (see, e.g.,
\cite[Theorem~4.4.4 and Proposition~4.6.2]{LL10})
$$R_{\mathrm{eff}}^{\SA}
(u, v) = \tfrac{1}{\pi} \log \|u-v\|_2 + O(1)\,.$$
Combined with Lemma~\ref{lem-green-function-2}, it yields that for all $u, v\in \SA$
$$ R_{\mathrm{eff}}^{\ttorus}(g(u), g(v))\geq (2ar_3\pi + 1 +o(1)) h^{-1} R_{\mathrm{eff}}^{\SA}(u, v) \,,$$
where we have used the fact that $R_{\mathrm{eff}}^{\SA}(u, v) \leq \frac{1+o(1)}{\pi} \log n = \frac{(1+o(1))h}{a\pi}$.
Applying~\eqref{eq-GFF-resistance} and Lemma~\ref{lem-sudakov-fernique}, we obtain that
$$
\E [\max_{v \in V_n} \eta_v] \geq \sqrt{2ar_3\pi + 1 +o(1)} h^{-1/2}\E [\max_{u \in \SA} \xi_u]\,.
$$
Combined with \cite[Theorem~2]{BDG2000} which states that
$\E [\max_{ u \in \SA} \xi_u] = (\sqrt{2/\pi} + o(1))\log n$, this yields the desired lower bound on
$\E [\max_{v \in V_n} \eta_v]$.
\end{proof}
{As $|E_n|=3 a n^2 \log n (1+o(1))$,}
upon combining Theorem~\ref{thm-Ding}
and Lemma~\ref{lem-GFF-torus},
we immediately obtain Theorem~\ref{thm-cover-time}.

\section{Upper bound on mixing time: large deviations for admissible types}
\label{sec:ubd-mix}

For the task of upper bounding $\tmix(\ttorus^\diamond,\delta)$ it 
suffices to compare the stationary law with a worst case initial one, 
for which purpose any non-random initial configuration will do. 
Further, since $\tmix(\ttorus,\delta)$ is only $O(n^2)$ 
(see \cite[Theorem~5.5]{LPW}), we can and shall
instead start for convenience 
at $X_0^\diamond$ having all lamps off and initial position 
uniformly chosen in $V_n$. Fixing $s' > s > \Psi(\phi)$ 
and using $s$ in the sequel for setting the various excursion types,
our goal is to show that the total-variation distance 
between the law of $X^\diamond_{s' \tcovp}$ and the
uniform law goes to zero as $n \to \infty$. To this end,
let $\widehat{\CU}_{s'} := \CU(s' \tcovp)$ denote
the subset of the vertices $V_n$ of $\ttorus$ 
not visited by $X$ up to time $s' \tcovp$,
with $\widehat{\CU}'_{s'}$
corresponding to a second, independent copy $X'$ of the \abbr{srw} on $\ttorus$.
Then, with $X_0$ uniformly distributed, the $L^2$-norm of the density of the
law of $X^{\diamond}_{s' \tcovp}$ with respect to the uniform law,
is $\E \big[ 2^{|\widehat{\CU}_{s'} \cap \,\widehat{\CU}'_{s'}|} \big]$ 
(see \cite[Proposition~3.2]{Miller-Peres}). Adapting the argument of
\cite[Lemma~3.1]{Miller-Peres}, it thus suffices to
find an event $\widehat{\mathcal G}$ measurable on the path of the \abbr{srw}
$X$ on $\ttorus$ up to time $s' \tcovp$, such that as $n \to \infty$
\begin{equation}
\label{eq:jason-yuval}
\p[\widehat{\mathcal G}] \to 1, \qquad \text{and} \qquad
\E \big[2^{|\widehat{\CU}_{s'} \cap \, \widehat{\CU}_{s'}'|}\,
{\bf 1}_{\widehat{\mathcal G}} {\bf 1}_{\widehat{\mathcal G}'}
\big] \to 1 ,
\end{equation}
where $\widehat{\mathcal G}'$ 
corresponds to the independent copy $X'$ of the \abbr{srw} on $\ttorus$.
Without $\widehat{\mathcal G}$ and 
$\widehat{\mathcal G}'$, 
the right side of \eqref{eq:jason-yuval} 
amounts to the $L^2$-convergence to $1$ of the relevant density. 
Only $L^1$-convergence is needed for the total-variation mixing
and using $\widehat{\mathcal G}$ helps eliminate some 
rare events that may dominate the second moment (see also the 
discussion immediately following \cite[Proposition~3.2]{Miller-Peres}).

To establish \eqref{eq:jason-yuval}, fixing a
large integer $M$ we set hereafter
\[
r := M r' := M^2 \,.
\] 
Note that for each  
$\underline{i} := (\underline{i}^{(1)},\underline{i}^{(2)}) 
\in \{0,\ldots,2r-1\}^3 \times \{0,1\}^3$ the points of 
\begin{equation}\label{def:Astar3D}
\SA^\star_{\threeD} (\underline{i}) := 
\big(\underline{i}^{(1)} + (2r\N)^3 \big) \cap \big( [0,n)^2 \times [0,h) - 2r \, \underline{i}^{(2)}\big)
\end{equation}
are at least $2r$ apart in $\ttorus$, whereas the union of the
$(4r)^3$ sub-lattices $\SA^\star_{\threeD} (\underline{i})$
covers $V_n$. Indeed, $\SA^\star_{\threeD}(\underline{i})$ keeps minimal distance $2r$ from all
faces that meet at the corner of $[0,n)^2 \times [0,h)$
indicated by $\underline{i}^{(2)}$, thereby  
assuring the stated $2r$-separation \emph{on the torus} 
(even when $2r$ does not divide $n$ or $h$). 
\begin{figure}[ht!]
\begin{center}
\includegraphics[scale=0.85]{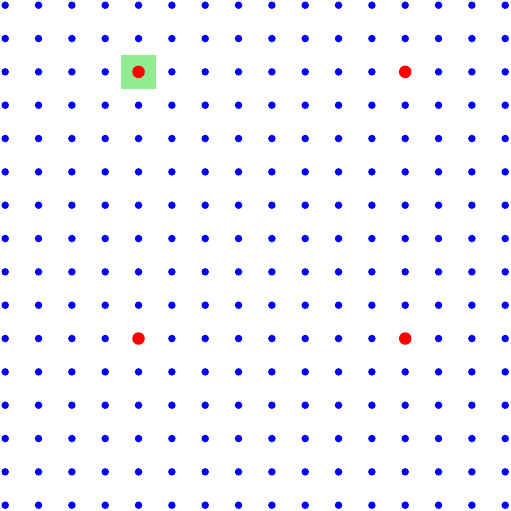}	
\end{center}
\caption{\label{fig:a2d} Illustration of a set $\SA^\star_{\twoD,k}(\underline{j}_k)$ as red dots of spacing $2R_k$ within a $\twoD$ sub-lattice of blue dots at
spacing $R_k''$. If $(x_1,x_2)$ is in the green square (of side length $R_k''$), then its center red point be $y_k(x)$. Here $R_k= 4 R_k''$ (that is, $M=2$).}
\end{figure}

Proceeding to 
produce in Definition \ref{def:Aij} the 
``2D-well-centered'' non-random 
subsets $\SA=\SA(\underline{i},\underline{j})$ 
of $\SA^\star_{\threeD} (\underline{i})$,
fix a large integer $L$ and approximate
the continuum of mesoscopic scales $h n^\rho$ by 
$R''_k = h [n^{\rho_k}]$ for $\rho_k=k/L$, $k=0,\ldots,L-1$ and 
$R''_L = [M^{-5} n]$.
Setting thereafter 
\[
R_k := M R'_k := M^2 R_k'' \,,
\] 
note that for any $L, M \ge 2$ and all $n$ large enough,
\begin{equation}\label{eq:sep-scale}
2 r < R_0'' < R_0' < R_0 < 2R_0 < R_1'' < R_1' < R_1 < 2R_1 < \cdots < R_L < n \,.
\end{equation}
Assuming hereafter that \eqref{eq:sep-scale} holds, for each  
$\underline{j}_k \in \{0,\ldots,(2R_k/R''_k)-1\}^2 \times \{0,1\}^2$
the points of  
\begin{equation}\label{def:Astar2Dk}
\SA^\star_{\twoD,k}(\underline{j}_k) := 
\big( R''_k \, \underline{j}^{(1)}_k + (2 R_k \N)^2 \big) \cap 
\big( [0,n)^2 - 2R_k \, \underline{j}_k^{(2)} \big)
\end{equation}
are $2R_k$ apart in the $2$D torus $\Z_n^2$ (thanks to the guard bands 
associated with $\underline{j}_k^{(2)}$), whereas for each $0 \le k \le L$ 
the union of $\SA^\star_{\twoD,k}(\underline{j}_k)$ over 
the $(4R_k/R''_k)^2$ possible values of $\underline{j}_k$ 
covers $\Z_n^2$. 
\begin{defn}\label{def:Aij}
For any $\underline{i}$ and 
$\underline{j}:=(\underline{j}_0,\underline{j}_1,\ldots,\underline{j}_L)$, 
let $\SA := \SA(\underline{i},\underline{j})$ 
denote the subset of those $x=(x_1,x_2,x_3) \in \SA^\star_{\threeD}(\underline{i})$ 
whose $2$D-projection $(x_1,x_2)$ lies for each $k=0,1,\ldots,L$
within the $R''_k$-sized square centered at some  
$y_k(x) \in \SA^{\star}_{\twoD,k}(\underline{j}_k)$.
\end{defn}
Note that $V_n$ is covered by the union of the 
\begin{equation}\label{def:kapp}
\kappa' := (4r)^3 (4M^2)^{2(L+1)}
\end{equation}
sets $\SA(\underline{i},\underline{j})$, with
$\kappa' = \kappa'(M,L)$ independent of $n$. 
We shall consider \eqref{eq:jason-yuval} for 
\begin{equation}\label{def:hatG}
\widehat{\CG} := \bigcap_{\underline{i},\underline{j}} \widetilde{\CG}_{\underline{i},\underline{j}}
\end{equation}
where each event $\widetilde{\CG}_{\underline{i},\underline{j}}$ on the path of 
the \abbr{srw} $X$ on $\ttorus$ up to time $s' \tcovp$
is defined via excursion counts associated with the 
points of

$\SA=\SA(\underline{i},\underline{j})$
(c.f. \eqref{def:Gij} and Definition 
\ref{dfn:Gz} for our specific choice
of $\widetilde{\CG}=\widetilde{\CG}_{\underline{i},\underline{j}}$). Then, 
by the union bound 
\[
\p[\widehat{\CG}^c]  \le
\kappa' \max_{\underline{i},\underline{j}} \p[\widetilde{\CG}_{\underline{i},\underline{j}}^c] 
\,.
\]
So, decomposing the set
$\widehat{\CU}_{s'} \cap \, \widehat{\CU}_{s'}'$ in the \abbr{rhs} of \eqref{eq:jason-yuval}
according to its intersections with the various $\SA(\underline{i},\underline{j})$,
by H\"older's inequality we get \eqref{eq:jason-yuval} upon showing that
for any $\underline{i}$, $\underline{j}$, as $n \to \infty$
\begin{equation}
\label{eq:js-yva}
\kappa' \p[\widetilde{\CG}_{\underline{i},\underline{j}}^c] \to 0 \qquad \text{and} \qquad
\E \big[2^{\kappa' |\SA(\underline{i},\underline{j}) \cap\,  
\widehat{\CU}_{s'} \cap \, \widehat{\CU}_{s'}'|}\,
{\bf 1}_{\widetilde{\CG}_{\underline{i},\underline{j}}} {\bf 1}_{\widetilde{\CG}'_{\underline{i},\underline{j}}} \big] \to 1 \,.
\end{equation}
Proceeding to prove~\eqref{eq:js-yva} for some fixed 
$(\underline{i},\underline{j})$ we avoid
crowded notations by omitting hereafter the specific
$(\underline{i},\underline{j})$ from all expressions.
In particular, given $(\underline{i},\underline{j})$,  
to each $x \in \SA=\SA(\underline{i},\underline{j})$ corresponds 
a unique vector $\underline{y} = (y_0,\ldots,y_L)$ of 
\emph{base points} $y_k = y_k(x) \in \SA^{\star}_{\twoD,k}$
(with $y_k(x)$ the closest point to  $(x_1,x_2)$ in $\SA^{\star}_{\twoD,k}$;
See Figure~\ref{fig:a2d} for an illustration of $\SA^\star_{\twoD,k}$ and $x \mapsto y_k(x)$).  We further let 
\begin{equation}\label{def:A2Dk}
\SA_{\twoD,k} :=
\{ y \in \SA^\star_{\twoD,k} : y = y_k(x) \hbox{ for some } x \in \SA \} \,,
\end{equation} 
denote the collection of all possible $k$-th level base points,
using the short notation $\SA_{\twoD}$, $R$, $R'$, $R''$
and $y(x)$ for $\SA_{\twoD,0}$, $R_0$, $R_0'$, $R_0''$
and $y_0(x)$, respectively.

\begin{figure}[ht!]
\begin{center}
\includegraphics[scale=0.85, page=1]{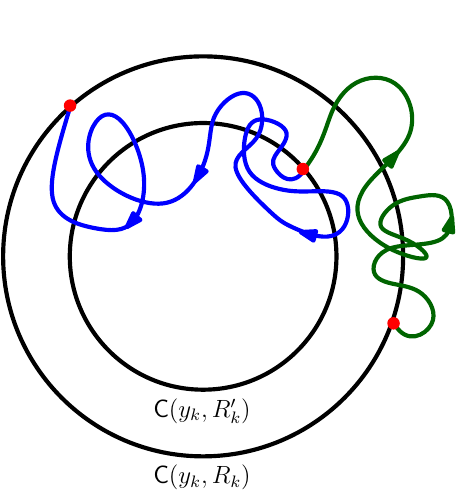}	
\end{center}
\caption{\label{fig:cyl_excursions} The $2D$ projection of an $R_k$-excursion 
of the random walk, from the boundary of a cylinder of radius $R_k$ back to itself via the boundary of a concentric cylinder of radius $R_k'$.  Indicated in dark green (resp.\ blue) is the external (resp.\ internal) part of the excursion.  
}
\end{figure}

Next, enumerating over $x \in \SA$ yields the disjoint $3$D-annuli 
of outer radius $r$ and inner radius $r'$, between 
the Euclidean balls $\ball(x,r)$ and $\ball(x,r')$ in $\ttorus$. 
For each $0 \leq k \leq L$, consider also the disjoint annuli 
of outer and inner radii
$R_k$ and $R'_k$, respectively, between
the cylinders $\cyl(y_k,R_k)$ and $\cyl(y_k,R_k')$ 
of height $h$ in $\ttorus$, based on the $2$D Euclidean
disks centered at $y_k \in \SA_{\twoD,k}$. 
As illustrated in Figure~\ref{fig:cyl_excursions},
\begin{figure}[ht!]
\begin{center}
\includegraphics[scale=0.85, page=2]{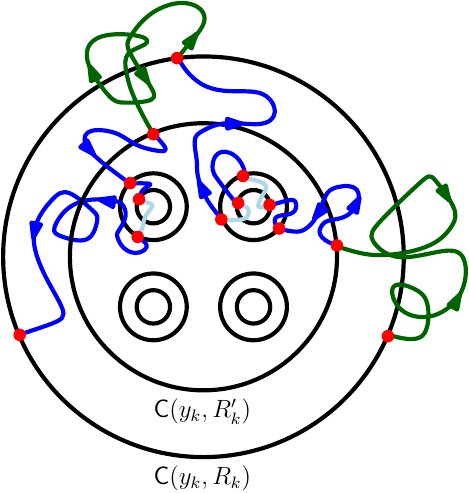}	
\end{center}
\caption{\label{fig:cyl_excursions_tree} The $R_k$-excursions across disjoint cylindrical annuli at different scales decompose into a tree structure, with
the internal part of any $R_{k-1}$-excursions (light blue), within the internal part of some $R_k$-excursion (blue). For well-separated annuli, 
the entrance and exit points of an $R_{k-1}$-excursion are approximately independent of the entrance and exit points of the parent $R_k$-excursion.}
\end{figure}
for any $k$, each cylindrical annulus decomposes the path
of the \abbr{srw} on $\ttorus$ into $R_k$-excursions. Each such excursion 
starts at the outer cylinder boundary and run until hitting the
inner cylinder boundary (which we call the excursion's
\emph{external part}), then goes back till exiting the
outer cylinder (called the excursion's \emph{internal part}).  
Note that for each $k$, conditional on their starting
and ending points, the internal parts of various
$R_k$-excursions of our collection of cylindrical annuli
are mutually independent of each other. For $n$ large enough
so \eqref{eq:sep-scale} holds, by the hierarchical structure of 
the sub-lattices 
$\SA^\star_{\twoD,k}$, the vector $\underline{y}$ associated with $x \in \SA$
is uniquely determined by $y(x)$. More generally,
each $R_{k-1}$-sized cylindrical annulus centered 
at $y \in \SA_{\twoD,k-1}$, $k \ge 1$, must be strictly inside 
$\cyl(y_k,R'_k)$ for \emph{some uniquely specified} $y_k \in \SA_{\twoD,k}$. 
Hence, as illustrated in 
Figure~\ref{fig:cyl_excursions_tree}, 
the $R_{k-1}$-excursions 
of the $y$-centered annulus decompose the internal parts of each of 
the $R_k$-excursions for the annulus centered at $y_k$. Similarly,
for $n$ large enough and $x \in \SA$, each $\ball(x,r)$ is 
strictly inside $\cyl(y(x),R')$, decomposing the internal
parts of each of the $R$-excursions of the cylindrical annulus
around $y(x)$, into what we call $r$-excursions (i.e., whose 
external part starts at $\partial \ball(x,r)$ and run till 
hitting $\ball(x,r')$, followed by the internal part up to 
the exit from $\ball(x,r)$). Here again, conditional on their starting
and ending points the internal parts of the various $r$-excursions
associated with the collection $\SA$ are independent of each
other.

As shown in Section \ref{subsec::typ-val}, 
\begin{equation}
\label{eq:typical-counts}
\cylE^\star(s): = 2 s \frac{(\log n)^2}{\log (R/R')} \quad\text{and}\quad
\ballE^\star(s) := \frac{4 s r'}{a} \log n ,
\end{equation}
are the typical counts of $R_k$-excursions and $r$-excursions, respectively,
by time $s \tcovp$. Utilizing these, we next summarize which
large deviations of the counts of cylindrical and ball excursions
around $x \in \SA$, are of concern in our proof of \eqref{eq:js-yva}.
We will show that \abbr{whp}, at least $\cylE^\star(s)$ of the $R_L$-excursions
around any $y_L \in \SA_{\twoD,L}$ are completed by time $s' \tcovp$. Hence, 
our concepts of a $z$-type point $x \in \SA$ and a $\underline{z}$-type 
$y(x) \in \SA_{\twoD}$, amount to having about $z^2 \, \ballE^\star(s)$ 
of the corresponding $r$-excursions around $x$, or respectively, 
having about $z_k^2 \, \cylE^\star(s)$ 
of the corresponding $R_k$-excursions around $y_k(x)$, 
$k=0,\ldots,L-1$, during the first $R_L$-excursions
around $y_L(x)$.
\begin{defn}\label{dfn:ztype}
Fix $s \in (\Psi(\phi),s')$ and small $\eta>0$ such that 
$1/\eta$ is integer.  
\begin{enumerate}
\item[(a)] For $\underline{z}=(z_0,\ldots,z_L)$ with
$z_k \leq z_L = 1$ and $z_k \in \eta \N$, $k=0,\ldots,L-1$,
we say that
$\underline{y}=(y_0,\ldots,y_L)$, or equivalently, that
$y_0 \in \SA_{\twoD,0}$, is \emph{of $\underline{z}$-type}
if the first $(z_k-2\eta)^2 \cylE^\star(s)$ of the  $R_k$-excursions
for the cylindrical annulus centered at $y_k$, are completed
within the first $\cylE^\star(s)$ $R_L$-excursions for
cylindrical annulus centered at $y_L$. In case $z_k < 1$ we 
further require that the first $(z_k-\eta)^2 \cylE^\star(s)$ are not
completed during these $R_L$-excursions. 
\item[(b)] Similarly, $x \in \SA$ is called \emph{of $z$-type}
(for  $z \in \eta \N$), if the first
$(z- 3 \eta)^2 \ballE^\star(s)$ of the $r$-excursions around $x$,
are completed within the first $\cylE^\star(s)$ $R_L$-excursions for
cylindrical annulus centered at
$y_L(x)$, where for $z<1$ we also require that the first
$(z-2\eta)^2 \ballE^\star(s)$ of those $r$-excursions are not completed
during said $R_L$-excursions.
\end{enumerate}
\end{defn}

Next, note that $\SA \cap \, \widehat{\CU}_{s'}$ is the disjoint union of
\begin{equation}\label{dfn:unvis-z}
\widetilde{\CU}_{s',\underline{z}} :=
\{ x \in \SA \cap \, \CU_{s'} : y(x) \mbox{ of } \underline{z}\mbox{-type} \},
\end{equation}
over the at most $\kappa_o = \eta^{-L}$ possible $\underline{z}$-types
induced on $\SA_{\twoD}$ by the \abbr{srw} $X$ on $\ttorus$. Likewise, 
$\SA \cap \, \widehat{\CU}'_{s'}$ is the disjoint union of the sets
$\widetilde{\CU}'_{s',\underline{z}'}$ defined in terms of the types
$\underline{z}'$ induced on $\SA_{\twoD}$ by the independent 
\abbr{srw} $X'$ on $\ttorus$. We set 
\begin{equation}\label{def:Gij}
\widetilde{\CG}:=\bigcap_{\underline{z}} \CG_{\underline{z}} \,,
\end{equation}
where each event $\CG_{\underline{z}}$ on the path of 
the \abbr{srw} $X$ on $\ttorus$ up to time $s' \tcovp$
is now associated with a specific choice of 
both $\SA=\SA(\underline{i},\underline{j})$ and $\underline{z}$
(see Definition \ref{dfn:Gz} below). Then, with
$\kappa := \kappa' \kappa_o^2$ 
for $\kappa'$ of \eqref{def:kapp}, 
similarly to our move from \eqref{eq:jason-yuval} to 
\eqref{eq:js-yva}, we get by the union bound and H\"older's inequality 
that~\eqref{eq:js-yva} holds provided that 
as $n \to \infty$, for any choice of 
$(\underline{i},\underline{j})$ and any two
types $\underline{z}$, $\underline{z}'$,
\begin{align}\label{eq:js-yv-lhs}
\kappa \p[{\CG}_{\underline{z}}^c] &\to 0 \,, \\
\E \big[2^{\kappa |\widetilde{\CU}_{s',\underline{z}} \cap\,
\widetilde{\CU}'_{s',\underline{z}'}|}\, &
{\bf 1}_{\CG_{\underline{z}}}
{\bf 1}_{{\CG'}_{\underline{z}'}} \big] \to 1 
\label{eq:js-yv-rhs}
\end{align}
(with ${\CG'}_{\underline{z}'}$ corresponding to the 
second, independent copy $X'$ of the \abbr{srw} on $\ttorus$).
We proceed to define the truncation events
$\CG_{\underline{z}}$ for \eqref{eq:js-yv-lhs}--\eqref{eq:js-yv-rhs}.
\begin{defn}\label{dfn:Gz}
For each $s < s'$, $\eta>0$ and type $\underline{z}$, let
${\CG}_{\underline{z}} = {\CG}_{\underline{z}}(s,\eta)$ be the event consisting of:
\begin{enumerate}
\item[(a)] By time $s' \tcovp$ the \abbr{srw} on $\ttorus$ completes
for each $R_L$-sized cylindrical annulus centered at $y_L \in \SA_{\twoD,L}$
the corresponding first $\cylE^\star(s)$ excursions.
\item[(b)] For $\rho_k=k/L$, $k=0,\ldots,L-1$, there are
at most $n^{2 b_{\rho_k}(z_k)}$ points $y_k \in \SA_{\twoD,k}$
to which corresponds some $y_0 \in \SA_{\twoD,0}$ of $\underline{z}$-type.
\item[(c)] If $x \in \SA$ is such that $y_0(x)$ is of
$\underline{z}$-type (cylindrical annuli), then 
for some $z \ge z_0$ the point $x$ is
also of $z$-type (in terms of $r$-excursions). 
\end{enumerate}
\end{defn}

From Definition~\ref{dfn:Gz}(b), we see that 
under the event $\CG_{\underline{z}}$ there is no $y(x)$ of 
$\underline{z}$-type, unless $b_{\rho_k}(z_k) \ge 0$ 
for all $0 \le k < L$. This is precisely the following
requirement \eqref{eq:z-trnc} that  
$\underline{z}$ be admissible (so it suffices to 
establish  
\eqref{eq:js-yv-rhs} only for admissible 
types $\underline{z}$, $\underline{z'}$).
\begin{defn}\label{dfn:adm-z}
Fixing $s \ge 1$, we say that a $\underline{z}$-type is admissible, if and only if  
\begin{equation}\label{eq:z-trnc}
\sqrt{s} \le \min_{k=0,\ldots,L-1} \; \Big\{ \frac{1-\rho_k}{1-z_k} \Big\} 
\end{equation}
for $\rho_k=k/L$, as in Definition \ref{dfn:Gz}.
\end{defn}

\noindent Denoting by $H_{x,z}$ the event of not hitting $x$ during the 
first $z^2 \ballE^\star(s)$ of the $r$-excursions of $X$ around $x$,
requirements (a) and (c) of Definition~\ref{dfn:Gz} imply 
that under the event $\CG_{\underline{z}}$ the set  
$\widetilde{\CU}_{s',\underline{z}}$ of \eqref{dfn:unvis-z} is a subset of 
\begin{equation}\label{dfn:unvis-z-ex}
\CU_{s,\underline{z}} := 
\{ x \in \SA : y(x) \mbox{ of } \underline{z}\mbox{-type}, \quad  
H_{x,z_0-3\eta} \mbox{ occurs}  \}
\end{equation}
(see also Definition~\ref{dfn:ztype} of $z$-type). Similarly,
$\widetilde{\CU}'_{s',\underline{z}'} \subseteq \CU'_{s,\underline{z}'}$
under the event $\CG'_{\underline{z}'}$. Hence, upon proving
\eqref{eq:js-yv-lhs} for $\CG_{\underline{z}}$ of Definition \ref{dfn:Gz},
it suffices to show that for  
any admissible $\underline{z}$-type and $\underline{z}'$-type,
as $n \to \infty$,
\begin{equation}\label{eq:js-yv}
\E \big[2^{\kappa |\CU_{s,\underline{z}} \cap\,
\CU'_{s,\underline{z}'}|}\,
{\bf 1}_{\CG_{\underline{z}}}
{\bf 1}_{{\CG'}_{\underline{z}'}} \big] \to 1 \,.
\end{equation}

\subsection{Variational formulas and admissible annuli profiles}\label{sec:ubd-gen-str}
We first establish the variational representations 
of Lemma \ref{lem:Psi}
for $\Psi(\phi)$ of~\eqref{eq:mix-bd} whose relevance to the asymptotic structure of $\CU(s \tcovp)$
has already been discussed in Section~\ref{subsec:pf-outline}. 
\begin{proof}[Proof of Lemma \ref{lem:Psi}]
First, set $\sh(\rho):= \sqrt{\rho(\phi+\rho/2)}$, $t:=\sqrt{s}$ and
\begin{equation}\label{eq:t-star}
t_\star = \sup_{\rho \in [0,1]} \{ \sh(\rho) + 1 -\rho \} \,.
\end{equation}
The conditions $b_\rho(z) \ge 0$ and $\alpha_\rho(z) \ge \rho$
are then
re-expressed as $t z \ge t - (1-\rho)$ and $t z \ge \sh(\rho)$, respectively.
So, with the optimal choice being $z=z_\star:=1-(1-\rho)/t$, it
follows that~\eqref{eq:var-pbm} holds if and only if $t \ge t_\star$.
That is, $\Psi(\phi) = t_\star^2$. Further, considering at $t=t_\star$
the optimal $z_\star = \sh(\rho)/(\sh(\rho)+1-\rho)$, yields the identity~\eqref{eq:t-lbd}. Finally, in~\eqref{eq:t-star}
the optimal choice is $\rho=\rho_\star = (\sqrt{2} -1) \phi$,
but in case $\phi \ge 1/(\sqrt{2}-1)$ it is out of range and one needs to settle instead for $\rho=1$. One easily checks that $\sh(\rho_\star)=\phi/\sqrt{2}$, while $\sh(1)=\sqrt{\phi+1/2}$,
hence with $t_\star$ monotone increasing
in $\phi$ it is easy to confirm from the preceding that
$t_\star^2 = \Psi(\phi)$ is given by the explicit formula
\eqref{eq:mix-bd}, as claimed.
\end{proof}

Denoting hereafter $\alpha_0(\cdot)$ of \eqref{eq:b-def}
by $\alpha(\cdot)$, we proceed with an 
analysis lemma that is key to the success of our scheme for bounding the exponential moments as in~\eqref{eq:js-yv}
for all admissible $\underline{z}$-types and $s > \Psi(\phi)$.
\begin{lem}\label{lem:calc}
Let $\Psi_{L,\eta}(\phi)$ denote, per given $L$ and $\eta$,
the minimal value of $s \ge 1$, such that if
type $\underline{z}$ is \emph{admissible} 
(see Definition \ref{dfn:adm-z}),
then for any $m=0,\ldots,L$,
\begin{equation}
\label{eq-def-gamma}
\gamma_{m,\eta} (\underline{z}) := \alpha (z_0-4 \eta)
- m \eta - \frac{1}{L} - \sum_{k=1}^m \Big[ \frac{1}{L}
- 2 s L ( z_k - z_{k-1} - 2\eta)_+^2\Big] \ge \eta.
\end{equation}
Then, with $\Psi(\cdot)$ given by the variational problem
\eqref{eq:var-pbm}, we have that
\begin{equation}
\label{eq:Psi-limit}
\Psi(\phi) = \limsup_{L \to \infty} \lim_{\eta \to 0} \{ \Psi_{L,\eta} (\phi) \} \,.
\end{equation}
\end{lem}
\begin{proof}
Recall that $z_L=1$ and note that the limit
$$
\Psi_L(\phi) := \lim_{\eta \to 0} \{\Psi_{L,\eta}(\phi)\} ,
$$
exists and corresponds to the requirement that
$\gamma_{m,0}(\underline{z}) \ge 0$ for $m=0,\ldots,L$
and admissible $\underline{z}$. Further, setting
$\Delta_k := t L ( z_k - z_{k-1})$, for $k=1,\ldots,L$
and $t:= \sqrt{s}$ we have from \eqref{eq:b-def} that
$$
\phi \, \alpha(z_0) = (t z_0)^2 = \Big(t - \frac{1}{L} \sum_{k=1}^L \Delta_k \Big)^2,
$$
yielding that $\sqrt{\Psi_L(\phi)}$ is merely
the infimum over all $t \ge 1$ such that for $m=0,\ldots,L$,
\begin{equation}\label{eq:PsiL}
(t - \frac{1}{L} \sum_{k=1}^L \Delta_k)^2 \ge
\phi \, \Big( \frac{m+1}{L} - \frac{2}{L} \sum_{k=1}^m (\Delta_k)_+^2 \Big) ,
\end{equation}
whenever $\underline{z} \in [0,1]^{L+1}$
satisfies~\eqref{eq:z-trnc}.
That is, denoting by $\mathcal{D}$ the collection of all
$\underline{\Delta} := (\Delta_1,\ldots,\Delta_L) \in \reals^L$ such that
\begin{equation}\label{eq:admis-del}
\delta_r := \frac{1}{L-r} \sum_{k=r+1}^L \Delta_k \in [0,1]\quad
\forall 0 \le r < L,
\end{equation}
we have that
$$
\sqrt{\Psi_L(\phi)} = \max_{m=0}^L \max_{\underline{\Delta} \in \mathcal{D}} \{t_m (\underline{\Delta}) \} ,
$$
with $t_m(\underline{\Delta})$ the smallest $t \ge 1$ for which
\eqref{eq:PsiL} holds, per given $m$ and $\underline{\Delta}$.

The value of $t_m(\underline{\Delta})$ depends only
on $\delta_m$ and $(\Delta_1,\ldots,\Delta_m)$. Further,
given $\delta_m$ and $\Delta := m^{-1}  \sum_{k=1}^m \Delta_k$, 
by Cauchy-Schwarz the maximal value of $t_m(\underline{\Delta})$ is attained
when $\Delta_k=\Delta$ for all
$1 \le k \le m$.
Thus, setting $\delta=\delta_m$, we deduce that
$\sqrt{\Psi_L(\phi)}$ is bounded above by
the minimal $t \ge 1$ such that
\begin{equation}\label{eq:PsiL-simp}
(t - (1-\rho) \delta - \rho \Delta)^2 \ge \phi \rho [1 - 2 (\Delta)_+^2] + \frac{\phi}{L} ,
\end{equation}
for any $\delta \in [0,1]$, $\Delta \in \reals$
and $\rho \in [0,1]$ for which $\rho L = m$ is integer valued.
Note that~\eqref{eq:PsiL-simp} trivially holds whenever $\Delta>1$ and $\rho>0$ (whereas for $\rho=0$ the
value of $\Delta$ is irrelevant).
Further, since $t \ge 1 \ge \rho, \delta \ge 0$, if
\eqref{eq:PsiL-simp} holds for $\Delta=0$, it also holds for any $\Delta<0$. Consequently, it suffices to consider
\eqref{eq:PsiL-simp} only for $\Delta,\delta \in [0,1]$.
Each choice of $(\Delta,\delta)$ in the latter range
corresponds to
$\underline{\Delta} = (\Delta,\ldots,\Delta,\delta,\ldots,\delta)$
in $\mathcal{D}$, hence we conclude that the right-side of~\eqref{eq:Psi-limit} \emph{equals} the minimal
$s=t^2 \ge 1$ satisfying
\eqref{eq:PsiL-simp} for all $\delta \in [0,1]$, $\rho \in (0,1]$
and $\Delta \ge 0$. To match this with~\eqref{eq:var-pbm}
we equivalently set $(1-\rho) \delta = t (1-w)$
and $\rho \Delta = t (w-z)$ with $1 \ge w \ge z$ such that
$b_\rho(w) \ge 0$ for $s=t^2$ (corresponding to $\delta \le 1$).
This transforms~\eqref{eq:PsiL-simp}, in terms of
$z$ and $w$, to the inequality
\begin{equation}\label{eq:la-under-2}
\alpha(z) + \frac{2 s(w-z)^2}{\rho} \ge \rho \,.
\end{equation}
Now, by elementary calculus we find that
\begin{equation}
\label{eq:alp-rho}
\alpha_\rho (w) =
\inf_{z \le w} \Big\{ \alpha(z) + \frac{2s (w-z)^2}{\rho}
\Big\}
\end{equation}
(with infimum attained at $z_\star := (2/\rho) w/(2/\rho+1/\phi)$).
Comparing the preceding with~\eqref{eq:var-pbm} we thus conclude
that~\eqref{eq:Psi-limit} holds, as claimed.
\end{proof}

\subsection{Tail behavior for admissible excursion counts}

Our approach to proving the upper bound in Theorem \ref{thm:cut-off}
is to establish \eqref{eq:js-yv-lhs} and \eqref{eq:js-yv} for
\begin{equation}\label{eq:sp-s-val}
s'= s + \epsilon = \Psi_{L,\eta}(\phi) + 2\epsilon,
\end{equation} 
when $n \to \infty$ followed by $M \to \infty$. As explained before,
this would imply that $\tmix \le (s'+o(1)) \tcovp$ and consequently, by
Lemma~\ref{lem:calc}, upon taking $\eta \downarrow 0$,
$L \to \infty$ and finally $\epsilon \downarrow 0$ we get that  
$\tmix \le (\Psi(\phi)+o(1)) \tcovp$. 

\smallskip
To this end, we use the following notation.
\begin{defn}\label{defn:NB}
Let $\ncylE_{y_k,k,j,w}$, for $k<j \le L$ and
$w \in [0,1]$ be the number of $R_k$-excursions for $y_k \in \SA_{\twoD,k}$,
completed during the first $w^2 \cylE^\star(s)$ $R_j$-excursions for
the corresponding $y_j \in \SA_{\twoD,j}$ (with $\ncylE_y := \ncylE_{y,0,L,1}$).  Let $\ncylE_{y_L,L}$ be the number of $R_L$-excursions around $y_L \in \SA_{\twoD,L}$ which are completed by time $s' \tcovp$. Next, for $x' \in \ball(x,R'')$ and $z \ge \eta$, let
$\nballE_{x,z}^{x'}$ be the number of $r$-excursions around $x \in \SA$
during the first $z^2 \cylE^\star(s)$ excursions of the $R_0$-cylindrical 
annulus centered at $x'$.

\end{defn}
As detailed in Section \ref{sec:pf-z-types}, both 
\eqref{eq:js-yv-lhs} and \eqref{eq:js-yv} follow from  
the next two lemmas, whose proofs are provided
in Sections~\ref{sec:3d} and~\ref{sec:2d}.
\begin{lem}
\label{lem:3d}
Fix $s>1 \ge z > \eta>0$. If $M \ge M_0(\eta,z)$ 
and $n \ge n_0(M)$, then  
\begin{equation}
\label{eq:bd3}
\p[H_{x,z}] \le n^{-\alpha(z-\eta)} \qquad \forall x \in V_n \,.
\end{equation}
Further, uniformly over $x \in V_n$ and 
$x' \in \ball(x,R'')$, as $n \to \infty$,
\begin{equation}
\label{eq:bd2}
 n^2 (\log n) \p \big[ \nballE^{x'}_{x,z}  < (z-\eta)^2 \ballE^\star(s) \big]
\to 0
\end{equation}
\end{lem}
\begin{Remark-lbl}\label{rmk-bd3}
The bound \eqref{eq:bd3} remains in effect when conditioned 
on $X_0=v$ and the start and end points of all $r$-excursions around $x$ 
(see Proposition \ref{prop::3d_tail_probability}). 
Similarly, from \eqref{eqn::too_few_3d_excursions}
the convergence in \eqref{eq:bd2} holds uniformly with 
respect to 
the position of $x$ within $\ball(x',R'')$ and
the start/end points of the $R$-excursions around $x'$.
\end{Remark-lbl}
\begin{lem}\label{lem:2d}
For any fixed $s',s>1$, any positive integer $L$,
$w, z \ge \widetilde{\eta} \ge 0$ and $L \ge j > k \ge 0$,
we have for all $M \ge M_1(\widetilde{\eta},z,w,j,k)$ large enough,
as $n \to \infty$, that uniformly over $y_L \in \SA_{\twoD,L}$
and $y_k \in \SA_{\twoD,k}$,
\begin{align}
\label{eq:bd4}
& n^{M} \p[|\ncylE_{y_L,L} - \cylE^\star(s')| \ge \widetilde{\eta}\, \cylE^\star(s') ]
\to 0 ,\\
\label{eq:bd1}
&
\limsup_{n \to \infty}
\Big| \frac{\log \p[ \ncylE_{y_k,k,j,w}(s) \le
(z-\widetilde{\eta})^2 \cylE^\star(s)]}{\log n} + \frac{2 s (w - z)_+^2}{\rho_j-\rho_k} \Big| \leq \widetilde{\eta} \,.
\end{align}
\end{lem} 
\begin{Remark-lbl}\label{rmk-bd2d}
See Proposition \ref{prop::typical_2d_excursion_count} 
which implies \eqref{eq:bd4}. In Section \ref{sec:2d} we further
show that \eqref{eq:bd1} holds uniformly 
in $x \in \SA$ with $y_k(x)=y_k$ (i.e., over the relative position 
of $y_k$ in the $R_j''$-sized square centered at $y_j=y_j(x)$), 
and uniformly with respect to the start/end points of the 
$R_j$-excursions around $y_j$.
\end{Remark-lbl}

\subsection{The proof of \eqref{eq:js-yv-lhs} and \eqref{eq:js-yv}.}\label{sec:pf-z-types}
First, as soon as 
$(1-\widetilde{\eta}) s' > s$ we deduce from \eqref{eq:bd4}
upon taking the union over the at most $M^6$ possible values of
$y_L$, that requirement (a) in Definition~\ref{dfn:Gz}
is satisfied with probability going to one as $n \to \infty$. 
Next, for $k<L$ let $Y_k$ denote the number of 
$y_k \in \SA_{\twoD,k}$ to which corresponds some $y_0 \in \SA_{\twoD,0}$ 
of $\underline{z}$-type.

If $z_k<1$ it follows by Definition~\ref{dfn:ztype} that necessarily 
$\ncylE_{y_k,k,L,1} \le (z_k-\eta)^2 \cylE^\star(s)$
for any such $y_k$.
With
$|\SA^{\star}_{\twoD,k}| \le \lceil n/(2R_k) \rceil^2 \le n^{2 - 2 \rho_k}$
upon considering \eqref{eq:bd1} for $j=L$, 
$\widetilde{\eta} = (\eta/2)^2$, $w=1$ and $z=z_k-\eta+\widetilde{\eta}$, we see 
that for $n$ large enough, $\E (Y_k) \le n^{2 b_{\rho_k}(z_k)-\widetilde{\eta}}$. Hence,
by Markov's inequality and union over $0 \le k < L$,
we deduce that 
Definition~\ref{dfn:Gz}(b) also holds with probability
going to one as $n \to \infty$ (the case $z_k=1$ trivially
holds by the preceding bound on $|\SA^{\star}_{\twoD,k}|$).
In particular, as soon as $s (1-4\eta)^2 > 1$, necessarily $z_0 \ge 5\eta$,
whereupon if $y_0(x)$ is of $\underline{z}$-type
and $x$ is not of $z$-type for some $z \ge z_0$, then
$\nballE^{x'}_{x,z_0-2\eta} < (z_0-3\eta)^2 \ballE^\star(s)$,
for $x':=(y_0(x),x_3) \in \ball(x,R'')$.
Combining \eqref{eq:bd2} at $z=z_0-2\eta$ with a union bound
over the at most $n^2 \log n$ points of $\SA$, we conclude
that Definition~\ref{dfn:Gz}(c) also 
holds with probability going to one as $n \to \infty$.
With $\kappa$ independent of $n$,
this establishes \eqref{eq:js-yv-lhs} for any $s'>s>1$ all 
$\eta>0$ small enough and every possible type $\underline{z}$.

Turning to deal with \eqref{eq:js-yv}, we may and shall fix 
$\epsilon>0$, $s,s'$ as in \eqref{eq:sp-s-val} and two
admissible types $\underline{z}$, $\underline{z'}$, where 
as mentioned before $z_0 \ge 5 \eta$ and $z'_0 \ge 5 \eta$. Next, 
for $0 \le k \le L$, let $J_k := |\Gamma_{\underline{z}} (k) \bigcap \Gamma_{\underline{z}'}'(k)|$, where  
$$
\Gamma_{\underline{z}}(k) := \{ y_{k} \in \SA_{\twoD,k}
\mbox { for some } \underline{y} \mbox { of } \underline{z}-\mbox{type}\} 
$$
and $\Gamma'_{\underline{z}'}(k)$
denoting the same sets for an independent \abbr{srw} $X'$ on $\ttorus$.
Recall \eqref{dfn:unvis-z-ex} that the image of
$\CU_{s,\underline{z}} \cap \CU'_{s,\underline{z}'}$
via $x \mapsto y_0(x)$ is a subset of the at most $J_0$
points from $\SA_{\twoD,0}$ having the corresponding types, where
to each $y \in \SA_{\twoD,0}$ correspond 
\begin{equation}\label{def:m}
|\{x \in \SA : y_0(x)=y\}| \le h^3 := m
\end{equation}
points from $\SA$. Given the
position of their starting and ending points, the $r$-excursions 
of \abbr{srw} $X$ around each $x \in \SA$, are
mutually independent and further independent of the random subset 
$\Gamma_{\underline{z}}(0) \subseteq \SA_{\twoD,0}$. Likewise, 
given their starting/ending points, the $r$-excursions of the 
\abbr{srw} $X'$ around each $x \in \SA$ are mutually independent 
and independent of $\Gamma'_{\underline{z}'}(0)$. Further, for
$x \in \SA$ with $y(x) \in \Gamma_{\underline{z}} (0) \bigcap \Gamma_{\underline{z}'}'(0)$ to be in 
$\CU_{s,\underline{z}} \cap \CU'_{s,\underline{z}'}$ we must 
have $H_{x,z_0-3\eta}$ occurring for $X$ 
and $H_{x,z_0'-3\eta}$ occurring for $X'$ (see \eqref{dfn:unvis-z-ex}). 
By \eqref{eq:bd3}, the probability of 
both events independently occurring at a given $x$, is at most
\begin{equation}\label{def:barp}
\bar{p} := n^{-\alpha(z_0-4\eta) - \alpha(z_0'-4\eta)} \,.
\end{equation}
By the uniformity of~\eqref{eq:bd3} per conditioning 
as in Remark~\ref{rmk-bd3}, we thus deduce from the preceding discussion that
\begin{equation}\label{def:S}
|\CU_{s',\underline{z}} \cap \CU_{s',\underline{z}'}'| 
\quad \hbox{ is stochastically dominated by } \;\;
 \sum_{\ell=1}^{J_0} \xi_\ell \,, 
\end{equation}
where $\xi_\ell$ are i.i.d.\ Binomial$(m,\bar{p})$ variables
independent of $J_0$, and $m$, $\bar{p}$ 
are given by 
\eqref{def:m} and \eqref{def:barp}, respectively.
Recall that $\kappa$ in \eqref{eq:js-yv-rhs}
is independent of $n$, while $\bar{p} \, m \to 0$ and $h^4/m \to \infty$ as
$n \to \infty$. Further, with
\begin{equation}\label{eq:bdum}
(1+u)^m \le 1 + e \, u\, m \qquad \hbox{ whenever} \qquad  u \, m \in [0,1] \,,
\end{equation}
we deduce that for all $n$ large enough, 
\begin{equation}\label{eq:binom-mgf}
\E [2^{\kappa \xi_\ell}] = [1 + (2^{\kappa}-1) \bar{p}]^m
\le 1 + h^4 \bar{p} \,.
\end{equation}
In view of \eqref{def:S} and \eqref{eq:binom-mgf}, 
\[
\E \big[2^{\kappa |\CU_{s,\underline{z}} \cap\,
\CU'_{s,\underline{z}'}|}\,
{\bf 1}_{\CG_{\underline{z}}}
{\bf 1}_{{\CG'}_{\underline{z}'}} \big] \le
\E \Big[ \prod_{\ell=1}^{J_0} 2^{\kappa \xi_\ell} \Big] \le 
\E \Big[ \big( 1+ h^4 \bar{p} \big)^{J_0} \Big] \,,
\]
with \eqref{eq:js-yv} holding as soon as
\begin{equation}\label{eq:bd5}
\E \Big[ \big( 1+ h^4 \bar{p} \big)^{J_0} \Big] \to 1\,.
\end{equation}
Turning to establish \eqref{eq:bd5}, note that
for any $k=0,\ldots,L-1$, given their starting
and ending points, the inner parts of the
$R_{k+1}$-excursions for different choices of
$y_{k+1} \in \SA_{\twoD,k+1}$ are
independent of each other, and of the random
subset $\Gamma_{\underline{z}}(k+1)$.
Thus, as in the preceding derivation,
the contributions
$\{\xi^\star_\ell, \ell=1,\ldots,J_{k+1}\}$
to $J_k$ that correspond to the possible
$y_{k+1} \in
\Gamma_{\underline{z}} (k+1) \bigcap \Gamma_{\underline{z}'}'(k+1)$, are
stochastically dominated by mutually
independent random variables $\{\xi_\ell\}$,
each having maximal size $m_k$ and mean
$m_k \bar{p}_k$, which are further
independent of $J_{k+1}$.
Here, $m_k:=n^{2(\rho_{k+1}-\rho_k)}=n^{2/L}$ bounds the
maximal number of points $y_k \in \SA_{\twoD,k}$
inside the $R_{k+1}$-cylinder
centered at some $y_{k+1} \in \SA_{\twoD,k+1}$. Further, if
$z_k < 1$ then  
$\ncylE_{y_k,k,k+1,w}(s) \le (z_k - \eta)^2 \cylE^\star(s)$
for $w = z_{k+1}-2\eta$ 
(compare Definitions \ref{dfn:ztype} and \ref{defn:NB}). 
Replacing $z_k<1$ by $z_k'<1$ and $w$ by $w'=z'_{k+1}-2\eta$, 
the same applies for the corresponding excursion counts induced by the 
\abbr{srw} $X'$. Considering 
the upper bound \eqref{eq:bd1} for $j=k+1$, $\widetilde{\eta}=\eta$ and
such values of $(w,z_k)$ and $(w',z'_k)$, recall Remark \ref{rmk-bd2d} that 
it holds uniformly over the relative position of $y_k$ in the 
$R''_j$-sized square centered at $y_j$ and with respect
to the start/end points of the $R_j$-excursions around $y_j$. 
Having here $\rho_j-\rho_k = 1/L$, we deduce  
by the independence of $X$ and $X'$ that for all $n$ large enough,
$$
\bar{p}_k := (n^{\eta - 2 s L(z_{k+1} -2 \eta - z_k)_+^2} \wedge 1) 
(n^{\eta - 2 s L (z_{k+1}'-2 \eta-z_k')^2_+} \wedge 1) \,.
$$
Each $\xi_\ell$ is no longer Binomial
(there are dependencies within each $R_{k+1}$-cylinder). Nevertheless, 
setting $u_{k+1} := e u_k \bar{p}_k m_k$ with $u_0:=h^4 \bar{p}$
we get inductively for $k=0,1,\ldots,L-1$, that if $u_k m_k \le 1$ then 
\begin{equation}\label{eq:iterate}
\E \Big[ \big( 1+ u_k \big)^{J_k} \Big]
\le \E \Big[ \prod_{\ell=1}^{J_{k+1}}  \E \big[ (1 + u_k)^{\xi_\ell}
\big] \Big]
\le \E \Big[ (1+ u_{k+1})^{J_{k+1}} \Big] 
\end{equation}
(utilizing stochastic domination, the mutual
independence of $\{J_{k+1}, \xi_\ell\}$ and finally
the inequality \eqref{eq:bdum} at $u_k$ and $\xi_\ell \le m_k$).

With both $\underline{z}$ and $\underline{z}'$ admissible, 
it follows by the definition of 
$\Psi_{L,\eta}(\phi)$ and $\gamma_{k, \eta}(\cdot)$ 
(c.f.~\eqref{eq-def-gamma}), that for any $s>\Psi_{L,\eta}(\phi)$, 
$$
u_k m_k = e^k u_0 m_0 \prod_{j=0}^{k-1} \bar{p}_j m_{j+1} \le \, 
e^k h^4 
n^{-\gamma_{k, \eta}(\underline{z})-\gamma_{k, \eta}(\underline{z}')} \,
\le e^k h^4 n^{-2 \eta} \to 0 
$$
when $n \to \infty$.
Hence, iterating \eqref{eq:iterate} over $0 \le k \le L-1$ yields that 
for $n \to \infty$,
$$
\E \Big[ \big( 1+ u_0 \big)^{J_0} \Big]
\le \E \Big[ \big( 1+ u_L \big)^{J_L} \Big]
\to 1 \,.
$$
Indeed, the latter convergence holds since $J_L \le |\SA_{\twoD,L}|$ 
is uniformly bounded (in $n$), whereas by the preceding, 
$u_L \to 0$ as $n \to \infty$.

\section{Proof of Lemma~\ref{lem:3d}: 3D-like tail probabilities}
\label{sec:3d}

\subsection{Evaluation of typical values}\label{subsec::typ-val}

Setting $R=M R'$, $R'=M R''$ and $R'' \geq h$ integer valued 
with both $M$ and $R''$ large enough,

we 
show that the typical excursion counts up to time $s \tcovp$ are given 
as in~\eqref{eq:typical-counts} by:
\begin{equation*}
\cylE^\star(s) := 2s \frac{(\log n)^2}{\log(R/R')} \quad\text{and}\quad \ballE^\star(s) := \frac{4s r'}{a} \log n.
\end{equation*}
To this end, we start with some basic results about the $2$D excursions.
In particular, \eqref{eqn::2d_excursion_concentration1} establishes 
\eqref{eq:bd4} and allows us to replace the random excursion counts 
$\ncylE_{\ul{y},L}(s)$ by their typical value $\cylE^\star(s)$, which 
by \eqref{eq::2d_excursion-means} and \eqref{eqn::2d_excursion_concentration2},
is also where 
the variables $\ncylE_{\ul{y},k}(s)$, $0 \leq k < L$, concentrate.
\begin{prop}
\label{prop::typical_2d_excursion_count}
Fix $\ul{y}=(y_0,\ldots,y_L)$ with $y_k \in \SA_{\twoD,k}$.  For $0 \le k \le L$, let $\ncylE_{\ul{y},k}(s)$ be the number of $R_k$-excursions for $\ul{y}$
completed during the first $\cylE^\star(s)$ of the
$R_L$-excursions for the corresponding $y_L \in \SA_{\twoD,L}$ with $\ncylE_{\ul{y},L}(s)$ denoting the number of latter $R_L$-excursions completed
by time $s \tcovp$.  Let $\cylE_{\ul{y},k}(s)$ denote
the expectation of
$\ncylE_{\ul{y},k}(s):=\ncylE_{y_k,k,L,1}$ in case $k<L$.  Then for each $\delta > 0$, there exists $C = C(\delta) > 0$ and $M(\delta)$ such that for all $M \geq M_0(\delta)$ there exists $n_0(\delta,M)$ such that for all $n \geq n_0(\delta,M)$ 
and $0 \le k \le L$, we have that

\begin{align}
(1-\delta) \cylE^\star(s) \leq \cylE_{\ul{y},k}(s)
&\leq (1+\delta) \cylE^\star(s) ,
\label{eq::2d_excursion-means}\\
\p\big[ |\ncylE_{\ul{y},L}(s) - \cylE_{\ul{y},L}(s)| \geq \delta \cylE_{\ul{y},L}(s) \big] &\leq \exp(-C s (\log n)^2) 
\label{eqn::2d_excursion_concentration1}\\
\p\big[ |\ncylE_{\ul{y},k}(s) - \cylE_{\ul{y},k}(s)| \geq \delta \cylE_{\ul{y},k}(s) \big] &\leq n^{-C s \delta^2} \label{eqn::2d_excursion_concentration2}
\end{align}
\end{prop}

\begin{proof}
Note that $\ncylE_{\ul{y},L}(s)$ counts the number of excursions between concentric 2D-disks of radii $R_L'$ and $R_L$ by the projected \abbr{srw} on $\Z_n^2$ during its first $\tfrac{4s}{\pi} n^2 (\log n)^2 (1+o(1))$ steps \cite{DPRZ-cover}.  
(As we explained earlier, the factor $2/3$ is due to the elimination of all vertical steps of the original \abbr{srw} on $\ttorus$.)
Our first assertion, namely~\eqref{eq::2d_excursion-means}
in the case $k=L$, thus follows from \cite[Lemma~3.2]{DPRZ-late}.  
That is, $\cylE_{\ul{y},L}(s)$ is up to leading order given by $\cylE^\star(s)$. Since $R_L/R'_L =M$ is independent of $n$, the bound 
\eqref{eqn::2d_excursion_concentration1} likewise follows from \cite[Lemma~3.2]{DPRZ-late}. Fixing $0 \leq k < L$ and considering \cite[Lemma~3.2]{DPRZ-late} for 
the $R_k$-excursions completed during the same number of 
steps by the projected \abbr{srw}, it further follows from 
\eqref{eqn::2d_excursion_concentration1}
that $\cylE_{\ul{y},k}(s) = \cylE_{\ul{y},L}(s) (1+o(1))$.   The same argument also gives~\eqref{eqn::2d_excursion_concentration2}.
\end{proof}

We proceed to establish the mean value 
of the relevant $3$D excursions. Hereafter, we let
$\sigma_\SW$ denote the first
exit time of the \abbr{srw} $\{X_k\}$
from a given $\SW \subseteq V_n$ using $\sigma^x_S$ for $\sigma_{\ball(x,S)}$
and the notation
$\ball'=\ball(x,r')$, $\ball=\ball(x,r)$,
$\cyl'=\cyl(x',R')$ and $\cyl=\cyl(x',R)$ for
balls of radii 
$r=M r' \le h$, $r'=M$ and cylinders, respectively, of any
centers $x,x' \in V_n$ with $|x-x'| \leq R''$.
\begin{prop}
\label{prop::k_excursions}
Suppose that $x,x' \in V_n$ with $|x-x'| \leq R''$. Then for each $\eta > 0$ there exists $M_0(\eta)$ such that for each $M \geq M_0(\eta)$ there exists $n_0(\eta,M)$ such that $n \geq n_0(\eta,M)$ implies that
\[
(1- \eta) z^2 \ballE^\star(s) \leq
 \E[\nballE_{x,z}^{x'}]
\leq (1+\eta) z^2 \ballE^\star(s)
\,.
\]
\end{prop}
\begin{proof}

Recall Definition \ref{defn:NB} that $\nballE_{x,z}^{x'}$ 
counts the \abbr{srw} excursions from $\partial \ball'$ to $\partial \ball$
during its first $z^2 \cylE^\star(s)$
excursions from $\partial \cyl'$ to $\partial \cyl$.
The latter $R$-excursions are conditionally independent given their 
starting and ending points. Hence, with $Z_\star$ counting the excursions 
that $X|_{[0,\sigma_\cyl]}$ makes from $\partial \ball'$ 
to $\partial \ball$, it suffices to show that 
$$
\E_v[Z_\star \giv X_{\sigma_\cyl}=w] = F_{\ball,\cyl} (1+o(1))
$$ 
(as $n \to \infty$ and $M \to \infty$),
uniformly in $v \in \partial \cyl'$ and $w \in \partial \cyl$, where
the nominal conversion factor from $R$-excursions to ball excursions is
\begin{equation}
\label{eqn::excursion_ratio}
F_{\ball,\cyl} := \frac{\ballE^\star(s)}{\cylE^\star(s)} = \frac{2r'}{h} \log(R/R').
\end{equation}
Indeed, we show in Lemma~\ref{lem::probability_hit_center} that 
\begin{equation}
\label{eq:hit-prob-center-rep}
\p_v [ \tau_{\ball'} < \sigma_\cyl \giv X_{\sigma_\cyl} = w] = F_{\ball,\cyl}
(1+o(1)),
\end{equation}
and from part (a) of Lemma~\ref{lem::k_excursions_after_hitting} we deduce 
that for $v' \in \partial \ball'$ 
\begin{equation}
\label{eqn::transient_excursion_bound-rep}
 \E_{v'} [ Z_\star  \giv X_{\sigma_\cyl} = w] \to 1 \quad\text{as}\quad n \to \infty \quad\text{then}\quad M \to \infty
\end{equation}
uniformly in $v'$ and $w$, which together complete the proof.
\end{proof}

Our next six lemmas culminate in Lemmas
\ref{lem::k_excursions_after_hitting} and
\ref{lem::probability_hit_center}, thereby completing 
the proof of Proposition~\ref{prop::k_excursions}. 
The first of these lemmas controls the fluctuations of 
positive harmonic functions in $\ttorus$. 
\begin{lem}
\label{lem::harnack_inequality}
Fixing $M \geq 2$ and $S=MS'$, we have that 
for all positive harmonic functions $f$ 
on the ball $\ball(0,S)$ in $\Z^3$,
\begin{equation}
\label{eqn::ball_harnack}
 \max_{u,u' \in \ball(0,S')} \frac{f(u)}{f(u')} = 1 + O(M^{-1}).
\end{equation}
Likewise, if $x \in V_n$, $S < n/2$, then for any $M \geq 2$ 
and every positive harmonic function $f$ 
on $\cyl(x,S)$ in $\ttorus$, we have that
\begin{equation}
\label{eqn::cylinder_harnack}
\max_{u,u' \in \ball(x,S')} \frac{f(u)}{f(u')} = 1+O(M^{-1}).
\end{equation}
\end{lem}
\begin{proof}
We first prove~\eqref{eqn::ball_harnack}.  The Harnack inequality \cite[Theorem~1.7.2]{lawler} implies that there exists a constant $C_0 > 0$ such that
\begin{equation}
\label{eqn::lawler_harnack}
 \max_{u,u' \in \ball(0,S/2)} \frac{f(u)}{f(u')} \leq C_0.
\end{equation}
It thus follows from \cite[Theorem~1.7.1]{lawler} that there exists a constant $C_1 > 0$ such that for any $u,u' \in \ball(0,S')$ we have
\begin{equation}
\label{eqn::lawler_difference}
 |f(u) - f(u')| \leq S' \frac{C_1}{S} \max_{v \in \ball(0,S/2)} f(v).
\end{equation}
Combining~\eqref{eqn::lawler_harnack} with~\eqref{eqn::lawler_difference} gives~\eqref{eqn::ball_harnack}.  Observe that~\eqref{eqn::cylinder_harnack} follows from~\eqref{eqn::ball_harnack} because any function which is harmonic on $\cyl(x,S)$ may be lifted to a harmonic function on a cylinder in $\Z^3$ with radius $S$ and periodic boundary conditions.
\end{proof}

Building on the preceding lemma, we next show that starting 
inside $\ball(x,S')$ any non-negative variable measurable 
on $X|_{[0,\sigma_{\ball(x,S')}]}$ is almost
independent of the \abbr{srw} on $\ttorus$ exit location of $\SW$ containing 
$\ball(x,S)$.
\begin{lem}
\label{lem::get_rid_of_conditioning}
Let $S=MS'$, $M \ge 2$ and $\wt{\ball}=\ball(x,S')$ for $x \in V_n$ and $S' \le h$. 
Suppose that $Z \geq 0$ is a random variable which depends only on $X|_{[0,\sigma_{\wt{\ball}}]}$.  Fix $\SW \subseteq V_n$ which 
contains $\ball(x,S)$.  Then we have that
\[ \max_{w,w' \in \partial \SW} \max_{u \in \wt{\ball}} \frac{\E_u[ Z \giv X_{\sigma_\SW} = w]}{\E_u[ Z \giv X_{\sigma_\SW} = w']} = 1+ O(M^{-1}).\]
In particular,
\[ \max_{w \in \partial \SW} \max_{u \in \wt{\ball}} 
\frac{\E_u[ Z \giv X_{\sigma_\SW} = w]}{\E_u[ Z]} = 1+ O(M^{-1}).\]
\end{lem}
\begin{proof}
Fix $u \in \wt{\ball}$ and $w \in \partial \SW$.  Then we have that
\begin{align}
       \E_u[ Z \giv X_{\sigma_\SW} = w]
&= \sum_{v \in \partial \wt{\ball}} \E_u[ Z \giv X_{\sigma_{\wt{\ball}}} = v] \p_u[ X_{\sigma_{\wt{\ball}}} = v \giv X_{\sigma_\SW} = w]. \label{eqn::conditioning_decomposition}
\end{align}
By Bayes' rule, we can write
\begin{equation}
\label{eqn::get_rid_of_conditioning_bayes}
 \p_u[ X_{\sigma_{\wt{\ball}}} = v \giv X_{\sigma_\SW} = w] = \frac{\p_u[ X_{\sigma_\SW} =w \giv X_{\sigma_{\wt{\ball}}} = v]}{\p_u[X_{\sigma_\SW} =w]} \p_u[X_{\sigma_{\wt{\ball}}} = v].
\end{equation}
By the strong Markov property, the ratio on the \abbr{rhs} of~\eqref{eqn::get_rid_of_conditioning_bayes} is contained in $[\kappa^{-1},\kappa]$ where
\begin{equation}
\label{eqn::get_rid_of_conditioning_ratio}
\kappa := \max_{v,v' \in \partial \wt{\ball}}\frac{\p_v[X_{\sigma_\SW} = w]}{\p_{v'}[ X_{\sigma_\SW} = w]}.
\end{equation}
Since $v \mapsto \p_v[X_{\sigma_\SW} = w]$ is harmonic on
$\ball(x,S)$, by Lemma~\ref{lem::harnack_inequality} 
we know that $\kappa = 1+O(M^{-1})$ uniformly in $w$.  Combining this with~\eqref{eqn::conditioning_decomposition} and using that $Z \geq 0$ implies the 
stated result.
\end{proof}

Using the preceding lemma, we
establish \eqref{eqn::transient_excursion_bound-rep} 
and further show that if $X_0$ is far from $x$, 
then $X|_{[0,\sigma_{\cyl}]}$ spends a negligible time in $\ball'$. 
To this end, we use hereafter 
\begin{equation}\label{dfn:loc}
\loc_{s,t}(\SW) := \sum_{k=s}^{t-1} \one_{\{X_k \in \SW\}} \,,
\end{equation}
for the \abbr{srw} local time of $\SW$ between
times $s \le t$, with $\loc_{t}(\SW) := \loc_{0,t}(\SW)$.
\begin{lem}
\label{lem::k_excursions_after_hitting}
Suppose that $x,x' \in V_n$ with $|x-x'| \leq R''$ and $w \in \partial \cyl$.
\begin{enumerate}[(a)]
\item There exists a universal finite constant $c_1$ such that 
starting at any $v' \in \partial \ball'$
the law of $Z_\star$ conditional on $\{X_{\sigma_\cyl} = w\}$ is stochastically dominated by $1+Y$ where $Y$ is a Geometric$(c_1/M)$ variable.

\item For $h'=h/(2M)$, uniformly in $v \in \partial \ball(x,h')$ and $w$,
\begin{equation}
\label{eqn::transient_local_time_bound}
\E_v[ \loc_{\sigma_\cyl}(\ball') \giv X_{\sigma_\cyl} = w] \to 0 \quad\text{as}\quad n \to \infty \quad\text{then}\quad M \to \infty .
\end{equation}
\end{enumerate}
\end{lem}
\begin{proof}
(a)  We first show that for some $C_1$ finite, 
any $u' \in \partial \ball$ and all $n, M$,
\begin{equation}
\label{eqn::exit_good_place}
\p_{u'}[ X_{\sigma_{h/4}^x} \notin 
\ball(x,\tfrac{h}{4}) \setminus \cyl(x,\tfrac{h'}{M})
\giv X_{\sigma_\cyl} = w] \leq \frac{C_1}{M}.
\end{equation}
Indeed, applying Lemma~\ref{lem::get_rid_of_conditioning} 
for $S'=h/4 \ge r$ and $\SW= \cyl \supset \ball(x,M S')$, we get 
\eqref{eqn::exit_good_place} upon noting that 
due to \cite[Lemma~1.7.4]{lawler},
\[
\p_{u'}\big[X_{\sigma_{h/4}^x} \notin
\ball(x,\tfrac{h}{4}) \setminus \cyl(x,\tfrac{h'}{M})
\big] \leq \frac{C_1}{M} \,.
\]
Similarly, upon applying Lemma~\ref{lem::get_rid_of_conditioning} 
for $Z={\bf 1}_{\{\tau_{\ball'} < \sigma_{h/4}^x\}}$, we 
can deduce from \cite[Theorem~1.5.4]{lawler} that 
for some universal $C_2$ finite
\begin{equation}
\label{eqn::hit_h_ball}
\p_{u'}[ \tau_{\ball'} < \sigma_{h/4}^x \giv X_{\sigma_\cyl} = w] \leq \frac{C_2}{M}.
\end{equation}
We next claim that for some $C_3(M)<\infty$ and all $u \in 
\ball(x,\tfrac{h}{4}) \setminus \cyl(x,\tfrac{h'}{M})$,
\begin{equation}
\label{eqn::return_from_good_place}
\p_u[ \tau_{\ball'} < \sigma_\cyl \giv X_{\sigma_\cyl} = w] \leq \frac{C_3(M)}{\log \log n}.
\end{equation}
Indeed, by Bayes' rule we can rewrite the \abbr{lhs} of~\eqref{eqn::return_from_good_place} as
\begin{equation}
\label{eqn::ratios}
\frac{\p_u[ X_{\sigma_\cyl} =w \giv \tau_{\ball'} < \sigma_\cyl]}{\p_u[ X_{\sigma_\cyl} = w]} \p_u[ \tau_{\ball'}  < \sigma_\cyl].
\end{equation}
By \cite[Exercise~1.6.8]{lawler}, the rightmost factor in~\eqref{eqn::ratios} is of order $C_3(M) / \log \log n$, so to complete the proof of~\eqref{eqn::return_from_good_place} it suffices to show that the left ratio in~\eqref{eqn::ratios} is uniformly bounded.  Applying the strong Markov property for the first time that $X$ hits $\partial \ball(x,\tfrac{h}{4})$ after $\tau_{\ball'}$, it in turn suffices to show that
\[ \max_{u,\wt{u} \in \partial \ball(x,\tfrac{h}{4})} \frac{\p_{\wt{u}}[ X_{\sigma_\cyl} =w]}{\p_{u}[ X_{\sigma_\cyl} = w]}\]
is bounded.  Such boundedness follows from \cite[Theorem~1.7.2]{lawler} since $u \mapsto \p_u[X_{\sigma_\cyl} = w]$ is harmonic.  Combining~\eqref{eqn::exit_good_place},~\eqref{eqn::hit_h_ball}, and~\eqref{eqn::return_from_good_place} yields the claimed stochastic domination of
the law of $Z$. 
\newline
(b) The same argument as in the proof of part (a) shows that here 
the number of excursions between $\ball'$ and $\ball(x,h')$ during the
time interval $[0,\sigma_\cyl]$ is stochastically dominated 
by a Geometric$(c(M)/\log h')$ for some finite $c(M)$.
Further, within each excursion between $\ball'$ and $\ball(x,h')$ 
we are in the setting of \abbr{srw} on $\Z^3$. Hence,
by a similar argument, relying once more on 
Lemma \ref{lem::get_rid_of_conditioning} and the relevant results from
\cite[Chapter 1]{lawler}, the expected contribution to $\loc(\ball')$ 
during such an excursion, conditional on its start/end points, is
uniformly bounded by $c'(M)$. Due to the independence
of these excursions given their start/end points, we thus 
deduce \eqref{eqn::transient_local_time_bound} by an 
application of Wald's identity.
\end{proof}

Turning to the proof of \eqref{eq:hit-prob-center-rep}, 
our next lemma gives a precise estimate of the Green's function 
for the \abbr{srw} on $\ttorus$ killed upon exiting $\cyl$ 
(conditioned on its exit location).  We note that for large 
$n$ and $M$ the 
resulting Green's function exhibits 
both $2$D (the term $\log (R/|v-x|)$) and $3$D (the factor $1/h$) 
behaviors.
\begin{lem}
\label{lem::green_kernel_conditioned_to_exit}
Suppose $x,x' \in V_n$ with $|x-x'| \leq R''$.  Let $\Green^w(v,x)$ denote the Green's function
for $X$ stopped upon hitting $\partial \cyl$ 
conditioned on exiting $\cyl$ at a given $w \in \partial \cyl$.
Then, for any $v \in \partial \cyl'$ and $\beta < 2$
\[ \Green^w(v,x) = \frac{3+O(M^{-1})}{\pi h}\big( \log R - \log|v-x|  + o(|v-x|^{-\beta}) + O(R^{-1})\big).\]
\end{lem}
\begin{proof}
Let $\tau_x$ be the first time that $X$ hits $x$, and let $\tau_x^+$ be the 
time of its first return to $x$. 
By the strong Markov property of $X$ at time $\tau_x^+$, we have
$$
\p_x[X_{\sigma_\cyl} = w \giv \tau_x^+ \leq \sigma_\cyl]=\p_x[X_{\sigma_\cyl} = w] \, ,
$$
i.e., the events $\{X_{\sigma_\cyl} = w\}$ and  $\{ \tau_x^+ \leq \sigma_\cyl\}$ are independent. Thus 
$$
\p_x[ \tau_x^+  > \sigma_\cyl \giv X_{\sigma_\cyl} = w]= \p_x[ \tau_x^+ > \sigma_\cyl] ; 
$$ 
taking reciprocals, $\Green^w(x,x) = \Green(x,x)$, where $G$ is the (unconditioned) Green's function for $X$ stopped upon hitting $\partial \cyl$. 

Applying the strong Markov property of 
$X$ conditioned on $\{X_{\sigma_\cyl} = w\}$,
at the stopping time $\tau_x$, we have that
\[ \Green^w(v,x)  = \p_v[ \tau_x \leq \sigma_\cyl \giv X_{\sigma_\cyl} = w] \Green^w(x,x).\]
By Bayes' rule, 

\begin{align*}
 \p_v[\tau_x \leq \sigma_\cyl \giv X_{\sigma_\cyl} = w]
 &= \frac{\p_v[ X_{\sigma_\cyl} = w \giv \tau_x \leq \sigma_\cyl]}{\p_v[X_{\sigma_\cyl} = w]} \p_v[\tau_x \leq \sigma_\cyl]\\
 &= \frac{\p_x[X_{\sigma_\cyl} = w]}{\p_v[X_{\sigma_\cyl} = w]} \p_v[\tau_x \leq \sigma_\cyl].
\end{align*}
Since $\Green(v,x) = \p_v[\tau_x \leq \sigma_\cyl]\Green(x,x)$,  combining the above we see that
\begin{equation}
\label{eqn::green_decomposition}
 \Green^w(v,x) = \frac{\p_x[X_{\sigma_\cyl} = w]}{\p_v[X_{\sigma_\cyl} = w]} \Green(v,x).
\end{equation}
Since $u \mapsto \p_u[X_{\sigma_C} =w]$ is harmonic
within $\cyl(x',R)$ and $v,x \in \cyl'$,
applying Lemma~\ref{lem::harnack_inequality} we arrive at
\begin{equation}
\label{eqn::green_no_exit_formula}
 \Green^w(v,x)= (1+O(M^{-1})) \Green(v,x).
\end{equation}
It thus remains only to estimate $\Green(v,x)$. To this end, let $\Green_{\Z_n^2}$ denote the Green's function associated with the projected (unconditioned) random walk in $\Z_n^2$ stopped upon exiting the disk of radius $R$ centered at $y(x')$.
Note that the projected random walk has a $1/3$ holding probability since this is the probability that the (unprojected) walk moves in the vertical direction.  
Let $\SW_x$ denote the collection of $h$ points in 
$V_n$ whose $2$D projection is equal to $y(x)$.  Then
\begin{equation}
\label{eqn::green_2d_3d_sum}
\Green_{\Z_n^2}(v,x) = \sum_{u \in \SW_x} \Green(v,u).
\end{equation}
Since $u \mapsto \Green(v,u)$ (for $v$ fixed) is harmonic for $u \neq v$, 
hence in $\cyl(x',R')$, whereas $\SW_x \subset \ball(x',2R'')$,
Lemma~\ref{lem::harnack_inequality} implies that
\begin{equation}
\label{eqn::green_ratio_bound}
\frac{\Green(v,u)}{\Green(v,u')} = 1+O(M^{-1}) \quad\text{for all}\quad u,u' \in \SW_x.
\end{equation}
Moreover, \cite[Proposition~1.6.7]{lawler} gives us that for every $\beta < 2$ we have
\[ \Green_{\Z_n^2}(v,x) = \frac{3}{\pi} \left( \log R - \log|v-x| \right) + o(|v-x|^{-\beta}) + O(R^{-1})\]
(recall the $1/3$ laziness).  Combining this with~\eqref{eqn::green_2d_3d_sum} and~\eqref{eqn::green_ratio_bound} tells us that for every $\beta < 2$ we have
\[ \Green(v,x) = \frac{(1+O(M^{-1}))}{h} \left( \frac{3}{\pi} \left( \log R - \log|v-x| \right) + o(|v-x|^{-\beta}) + O(R^{-1}) \right).\]
Combining this with~\eqref{eqn::green_no_exit_formula} gives the result.
\end{proof}

We are now going to estimate the expected amount of time that \abbr{srw} starting from $\partial \ball'$ spends in $\ball'$ before exiting $\cyl$.  
This estimate allows us to establish \eqref{eq:hit-prob-center-rep}
in the subsequent lemma.

\begin{lem}
\label{lem::time_in_ball}
For $x,x' \in V_n$ with $|x-x'| \leq R''$ any
$v' \in \partial \ball'$ and $w \in \partial \cyl$, let
\begin{equation*}
\oloc^{v',w}(\ball' ; \cyl)  = \E_{v'}[ \loc_{\sigma_\cyl}(\ball') \giv X_{\sigma_\cyl} = w]
\end{equation*}
(for $\loc_t(\cdot)$ as in \eqref{dfn:loc}). Then,
\begin{equation}
\label{eqn::time_in_ball}
\frac{\oloc^{v',w}(\ball' ; \cyl)}{2 (r')^2} \to 1 \quad\text{as}\quad n \to \infty \quad\text{then}\quad r'= M \to \infty.
\end{equation}
\end{lem}
\begin{proof}
We first reduce~\eqref{eqn::time_in_ball} to a computation which involves 
only the transient \abbr{srw} $\wt{X}$ on $\Z^3$ starting at $\wt{X}_0=v'$.  To this end note that for $h'=h/(2M)$,
\[ 
\oloc^{v',w}(\ball' ; \cyl) = \E_{v'}[ \loc_{\sigma_{h'}^x}(\ball') \giv X_{\sigma_\cyl} = w] 
+ \E_{v'}[ \loc_{\sigma_{h'}^x,\sigma_\cyl}(\ball') \giv X_{\sigma_\cyl} = w],
\]
and from part (b) of Lemma~\ref{lem::k_excursions_after_hitting} 
the right most term is $o(1)$ as $n \to \infty$ followed by $M \to \infty$.
Further, the other term on the \abbrs{rhs} involves a  variable of the type considered in Lemma~\ref{lem::get_rid_of_conditioning} 
for $S'=h'$. With $\cyl \subset \ball(x,h/2)$ it is thus within a uniform
$1+O(M^{-1})$ factor of $\E_{v'}[ \loc_{\sigma_{h'}^x}(\ball')]$, which 
is precisely the local time in $\ball'$ of $\wt{X}$ till its exit time of
$\ball(x,h')$. Let $\wt{Z}$ be the \emph{total} 
local time of $\wt{X}$ in $\ball'$, noting that since $h'=\Theta(\log n)$ 
while $r'=M$, it follows from \cite[Theorem~1.5.4]{lawler} that as $n \to \infty$,
\begin{equation}
\label{eqn::time_in_ball_3d}
 \E_{v'}[ \loc_{\sigma_{h'}^x}(\ball')] = \E_{v'}[\wt{Z}] + O(1).
\end{equation}
From \cite[Theorem~1.5.4]{lawler}, we have moreover that 
\begin{equation}
\label{eqn::time_in_ball_int_rep}
\frac{1}{(r')^2}\E_{v'}[\wt{Z}] \to c_3 \int_{\ball(0,1)} \frac{du}{|u-e_3|} 
\quad\text{as}\quad r' \to \infty,
\end{equation}
where $c_3:=3/(2\pi)$ is given explicitly in
\cite[Theorem~4.3.1, top of page 82]{LL10},
$e_3=(0,0,1)$ and 
$\ball(0,1) = \{v \in \R^3 : |v| < 1\}$ is
the unit ball in $\R^3$ with Lebesgue measure denoted by $du$; 
we note that an additional factor of $r'$ appears in the
normalization from spatially re-scaling. 
This convergence is uniform in $v' = \wt{X}_0$
and the proof is completed by finding after the change of coordinates 
$u=(t \cos \varphi \cos \theta,t \cos \varphi \sin \theta, 1 - t \sin \varphi)$ 
that the integral on the \abbr{rhs} of \eqref{eqn::time_in_ball_int_rep} 
is precisely $4 \pi/3$.
\end{proof}

Combining Lemmas~\ref{lem::green_kernel_conditioned_to_exit} and \ref{lem::time_in_ball}
we now establish \eqref{eq:hit-prob-center-rep}.
\begin{lem}
\label{lem::probability_hit_center}
Uniformly in $x,x' \in V_n$ with $|x-x'| \leq R''$,

$v \in \partial \cyl'$ and $w \in \partial \cyl$,
in the limit $n \to \infty$ followed by $M \to \infty$, 
\begin{equation}\label{eq:hit-prob-center}
\p_v[ \tau_{\ball'} < \sigma_\cyl \giv X_{\sigma_\cyl} = w] =
\frac{2r'}{h} \log (R/R') (1+o(1)).
\end{equation}
\end{lem}
\begin{proof} Recall that if $Z \ge 0$ and $\p[Z > 0] > 0$ 
then $\p[Z > 0] = \E[Z] / \E[Z\giv Z> 0]$.  
Applying this identity for $Z=\loc_{\sigma_\cyl}(\ball')$ 
conditional to $X_0=v$ and $X_{\sigma_\cyl}=w$, yields
\[ \p_v[ \tau_{\ball'} < \sigma_\cyl \giv X_{\sigma_\cyl} = w] 
= \frac{\oloc^{v,w}(\ball' ; \cyl)}{\hloc^{v,w}(\ball' ; \cyl)}, \]
where
\begin{align*}
 \oloc^{v,w}(\ball' ; \cyl) &:= \E_v[ \loc_{\sigma_\cyl}(\ball') \giv X_{\sigma_\cyl} = w] \\
 \hloc^{v,w}(\ball' ; \cyl) &:= \E_v[ \loc_{\sigma_\cyl}(\ball') \giv
X_{\sigma_\cyl} = w, \tau_{\ball'} < \sigma_\cyl].
\end{align*}
We thus arrive at \eqref{eq:hit-prob-center} by showing that uniformly 
in $x,x',v,w$ as $n \to \infty$ followed by $r'=M \to \infty$,
\begin{align}\label{eq:jason1-pf}
 \oloc^{v,w}(\ball' ; \cyl) &\sim \frac{4 (r')^3}{h} \log (R/R') \quad\text{and}
 \\
 \hloc^{v,w}(\ball' ; \cyl) &\sim 2 (r')^2.
 \label{eq:jason2-pf}
\end{align}
Note that by definition 
\[
\oloc^{v,w}(\ball' ; \cyl) = \sum_{u \in \ball'} G^w(v,u), 
\]
for the Green's function $G^w(\cdot,\cdot)$ 
of Lemma~\ref{lem::green_kernel_conditioned_to_exit}.
The estimate for $G^w(\cdot,\cdot)$ given there implies that 
uniformly in $u \in \ball'$ and $v \in \partial \cyl'$,
\[
G^w(v,u)=\frac{3}{\pi h}\log(R/R')(1+o(1))
\]
when $n \to \infty$ followed by $M \to \infty$ (so that $|v-u| \sim R'$).
Since $\ball'$ has to leading order $\frac{4 \pi}{3} (r')^3$ points, this
yields the stated formula~\eqref{eq:jason1-pf} for $\oloc^{v,w}(\ball' ; \cyl)$. 
Further, 
\[
\hloc^{v,w}(\ball' ; \cyl) = \E_v \big[
\oloc^{X_{\tau_{\ball'}},w}(\ball') \giv
X_{\sigma_\cyl}= w, \tau_{\ball'} < \tau_{\cyl} \big] ,
\]
and with $X_{\tau_{\ball'}} \in \partial \ball'$ we 
get~\eqref{eq:jason2-pf} by the uniform in $v'$ asymptotics of 
Lemma~\ref{lem::time_in_ball}.
\end{proof}

\subsection{Tail probabilities for 3D type events}
In this section we establish tail probabilities for $3$D type events, 
which imply~\eqref{eq:bd3} and~\eqref{eq:bd2} in 
the strong sense of Remark~\ref{rmk-bd3}.  
We start with the proof of~\eqref{eq:bd3}.
\begin{prop}
\label{prop::3d_tail_probability}
Fix $x \in V_n$ and let $\CF_{\ball'}$ be the $\sigma$-algebra generated by the 
entrance and exit points of all the excursions 
of $X$ from $\partial \ball'$ to $\partial \ball$.  
For any $s>1$, $1 \ge z > \eta > 0$ there exists $M_0$ such that for every $M \geq M_0$ there exists $n_0 = n_0(M)$ such that $n \geq n_0$ implies that a.s.
\begin{equation}
\label{eqn::3d_tail_probability}
  n^{-\alpha(z+\eta)} \leq  \p_v[ H_{x,z} \giv \CF_{\ball'}] \leq n^{-\alpha(z-\eta)} \,.
\end{equation}
The upper bound holds for all $x,v \in V_n$, with 
$|v-x|>R'$ for the lower bound.
\end{prop}

In order to prove Proposition~\ref{prop::3d_tail_probability}, we 
first estimate the probability that a \abbr{srw} 
starting from the boundary of a ball hits the center 
before exiting a larger ball, conditional on its exit point.
\begin{lem}
\label{lem::3d_success_probability}
Uniformly over $x \in V_n$, $v' \in \partial \ball'$ and $w \in \partial \ball$,
\begin{equation}\label{eq::3d_success}
\p_{v'} [ \tau_x < \sigma_{\ball}\giv X_{\sigma_{\ball}} = w] = (1+O(M^{-1})) \Delta,
\end{equation}
where for $c_3:=3/(2\pi)$ from \cite[Theorem~1.5.4]{lawler} (see \eqref{eqn::time_in_ball_int_rep}),
and $q$ of \eqref{eqn::r3_def},
\begin{equation}
\label{eqn::3d_success_probability}
\Delta = \frac{c_3 q}{r'}.
\end{equation}
\end{lem}
\begin{proof} By Bayes' rule, 
\[ \p_{v'}[ \tau_x < \sigma_{\ball}\giv X_{\sigma_{\ball}} = w]  = \frac{\p_{v'}[X_{\sigma_{\ball}} = w \giv \tau_x < \sigma_{\ball}]}{\p_{v'}[X_{\sigma_{\ball}} = w]} 
\p_{v'}[ \tau_x < \sigma_\ball].\]
By the strong Markov property of $X$ at $\tau_x$ the ratio on the \abbrs{rhs} is
\[ 
\frac{\p_x[X_{\sigma_\ball} = w]}{\p_{v'}[X_{\sigma_\ball} = w]} = 1+O(M^{-1})
\]
(where we used once again Lemma~\ref{lem::harnack_inequality}
for $S'=r'$ and $u \mapsto \p_u[X_{\sigma_\ball}=w]$ harmonic on $\ball$).
Let $\wt{X}$ denote the \abbr{srw} on $\Z^3$ starting at $v'$ and 
$\wt{\tau}_x$, $\wt{\sigma}_{\ball}$ be the corresponding stopping times. Then,
\begin{equation}
\label{eqn::3d_hit_point_before_ball}
 \p_{v'}[\tau_x < \sigma_{\ball}] =
1 - \frac{\p_{v'}[\wt{\tau}_x = \infty]}{ \p_{v'}[\wt{\tau}_x=\infty \giv \wt{\tau}_x \geq \wt{\sigma}_{\ball}]}.
\end{equation}
By \cite[Proposition~6.5.1]{LL10} 
(having same constant $c_3$ as in \cite[Theorem~1.5.4]{lawler}),
\begin{equation}\label{eq:z3-est-unc}
\p_{v'}[\wt{\tau}_x = \infty] \sim 1- \frac{c_3 q}{r'}.
\end{equation}
Applying the strong Markov property at $\wt{\sigma}_{\ball}$, we similarly
find that 
\begin{equation}\label{eq:z3-est-con}
\p_{v'}[\wt{\tau}_x=\infty \giv \wt{\tau}_x \geq \wt{\sigma}_{\ball}] \sim 1-\frac{c_3 q}{r}.
\end{equation}
Combining \eqref{eqn::3d_hit_point_before_ball}--\eqref{eq:z3-est-con} 
yields the stated estimate $\Delta (1+O(M^{-1})$ in \eqref{eq::3d_success}.
\end{proof}

\begin{proof}[Proof of Proposition~\ref{prop::3d_tail_probability}]
If $v \in \ball'$, we only reduce the event $H_{x,z}$ 
by shifting $v$ to the induced (random) first exit of $X$ from $\ball'$. 
Proceeding hereafter with $v \in V_n \setminus \ball'$
the inner parts of the $r$-excursions of $X$ around $x$
are independent of each other given $\CF_{\ball'}$. Thus, the conditional probability considered in~\eqref{eqn::3d_tail_probability} is the product of $z^2 \ballE^\star(s)$ probabilities.  Lemma~\ref{lem::3d_success_probability} implies the existence of
$\delta =\delta(M) \downarrow 0$ as $M \to \infty$ such that each of these probabilities is at most $(1-\Delta+\delta)$, 
uniformly in the initial and terminal points of the excursion. 
In view of \eqref{eq:typical-counts} and \eqref{eqn::3d_success_probability},
\[ 
(1-\Delta)^{z^2 \ballE^\star(s)} \le \exp(-\Delta z^2 \ballE^\star(s)) 
= n^{-\alpha(z)}.
\]
The stated upper bound follows since $\alpha(z-\eta) < \alpha(z)$.  
The complementary lower bound is similarly proved for $v \notin \ball'$.
\end{proof}

We now turn to establish~\eqref{eq:bd2}.
\begin{prop}
\label{prop::3d_excursion_count}
Fix $x' \in V_n$ and let $\CF_{\cyl'}$ be the $\sigma$-algebra generated by the 
entrance and exit points of all the excursions 
of $X$ from $\partial \cyl'$ to $\partial \cyl$.
For any $s>1 \ge z > \eta > 0$ there exists $\gamma > 0$ such that 
for all $n, r' \in \N$ large enough and every
$x \in \ball(x',R'')$, $v \in V_n \setminus \cyl'$, 
we have that a.s.
\begin{equation}
\label{eqn::3d_excursion_count}
 \p_v
 \Big[ \frac{\nballE_{x,z}^{x'}}{\ballE^\star(s)} \notin [(z-\eta)^2, (z+\eta)^2] \,\giv \CF_{\cyl'} \Big]  \leq n^{-\gamma r'}.
\end{equation}
\end{prop}
\begin{proof}
Fixing $s > 1 \ge z >\eta > 0$ we first show that for some 
$\gamma > 0$ all $n,r' \in \N$ large enough and
every $|x-x'| \le R''$, $v \in V_n$,
\begin{equation}
\label{eqn::too_few_3d_excursions}
\p_v \big[\nballE_{x,z}^{x'} < (z-\eta)^2 \ballE^\star(s) \giv \CF_{\cyl'} 
\big] \leq n^{-\gamma r'}.
\end{equation}
Indeed, $R'' + r < R'$ hence $\ball \subseteq \cyl'$
for all $n$ large enough. When $v \in \cyl'$ we thus may only 
reduce $\nballE_{x,z}^{x'}$ upon using the strong Markov 
property at the first exit of $\cyl'$. Consequently, it 
suffices to establish \eqref{eqn::too_few_3d_excursions}
for $v \notin \cyl'$. In the latter case, by Lemma~\ref{lem::probability_hit_center} 
there exist $\delta=\delta(M) \downarrow 0$ as $M \to \infty$ and 
$n_0 = n_0(M)$ such that for all $n \geq n_0$ the number $Z_\star$ 
of excursions from $\partial \ball'$ to $\partial \ball$ 
within one excursion from $\partial \cyl'$ to 
$\partial \cyl$ is stochastically bounded \emph{below} by a 
Bernoulli($p_n$) variable $J$ with $p_n = (1-\delta) F_{\ball,\cyl}$,
uniformly in $x,x'$ as stated and in 
the initial and terminal points of the excursion.  
Letting $N := z^2 {\ballE}^\star(s)$ and $N' := z^2 {\cylE}^\star(s)$,
the probability considered 
in~\eqref{eqn::too_few_3d_excursions} is thus 
bounded above by 
\[
P_\star:=\P(\sum_{i=1}^{N'} J_i \le (1-\eta/z)^2 N),
\] 
for i.i.d.\ $\{J_i\}$. 
From the definition of $F_{\ball,\cyl}$ we have that 
$N' = N (1-\delta)/p_n$ hence by Markov's inequality 
we deduce that for any $\theta > 0$, 
\begin{align}\label{eq:bd-pstar}
\frac{1}{N} \log(P_\star)
&\leq \theta (1-\eta/z)^2 + \frac{1-\delta}{p_n} \log 
\Big(1 - p_n(1-e^{-\theta}) \Big) .
\end{align}
The function $f(\kappa,\theta) := \theta - \kappa (1-e^{-\theta})$ 
decreases in $\kappa$ and is strictly negative for any 
$\kappa>1$ and $\theta>0$ small enough. Since 
$p_n \to 0$ as $n \to \infty$, the \abbr{rhs} of \eqref{eq:bd-pstar} 
converges to $\kappa^{-1} f( (1-\delta) \kappa,\theta)$, 
where $\kappa=(1-\eta/z)^{-2} >1$. With $\delta(M) \to 0$, there exists 
$\gamma'=\gamma'(\kappa)>0$ such that using $\theta > 0$ sufficiently small 
we get from \eqref{eq:bd-pstar} that for all $M \ge M_1$ and $n \ge n_1$
$$
P_\star \le e^{-\gamma' N} = n^{-\gamma r'}.
$$
Note that, in view of 
\eqref{eq:typical-counts}, the value of 
$\gamma=\frac{4 s}{a} \gamma' z^2  > 0$ is independent of $r'$. A similar argument shows that, by possibly decreasing $\gamma = 
\gamma(s,z,\eta) > 0$, for $v \notin \cyl'$ one has 
\begin{equation*}
\p_v \big[\nballE_{x,z}^{x'} > (z+\eta)^2 \ballE^\star(s) \giv \CF_{\cyl'} 
\big] \leq n^{-\gamma r'}.
\end{equation*}
Indeed, the only difference is that now we need to replace the i.i.d.\ copies of Bernoulli($p$) by i.i.d.\ copies of the product of Bernoulli $\widetilde{J}$ of mean $(1+\delta) F_{\ball,\cyl}$ and $1+Y$ for the Geometric random variable $Y$ of success probability $c_1/M$ as established in 
part (a) of Lemma~\ref{lem::k_excursions_after_hitting}.
\end{proof}

Further, combining Propositions~\ref{prop::3d_tail_probability} and~\ref{prop::3d_excursion_count} we obtain the following.  
\begin{prop}
\label{prop::prob_hit_point}
For $s> 1 \ge z \geq \eta > 0$, let 
$\wh{H}^{x'}_{x,z}$ be the event of not hitting 
$x$ during the first $z^2 \cylE^\star(s)$ excursions 
from $\partial \cyl'$ to $\partial \cyl$. Then, 
there exist finite $n_0=n_0(M)$, 
$M \ge M_0$, such that for every $n \geq n_0$, $x' \in V_n$, 
$x \in \ball(x',R'')$ and $v \in V_n \setminus \cyl'$ we have a.s.
\[ 
n^{-\alpha(z+\eta)} \leq 
\p_v [\wh{H}^{x'}_{x,z} \giv \CF_{\cyl'}]
\leq n^{-\alpha(z-\eta)}.
\]
\end{prop}

\section{Proof of Lemma~\ref{lem:2d}: 2D excursion counts at various radii}
\label{sec:2d}

This section is devoted to the proof of~\eqref{eq:bd1}.
To this end, recall our notations of
$R'' = h$, $R = M^2 h$ and for
any fixed $L \in \N$ and $k \in \{0,\ldots,L-1\}$, having $\rho_k = k/L$ and 
$R_k = R \, [n^{\rho_k}]$, while $R_L= [n/M^{5}] \, M^2$.
Fixing $w,z$ and $j \in \{k+1,\ldots,L\}$ we let $\ncylE_{y_k,k,j,w}(s)$ 
as in Definition \ref{defn:NB}
count the number of $R_k$-excursions for $y_k \in \SA_{\twoD,k}$ completed during the $w^2 \cylE^\star(s)$ first $R_j$-excursions for the corresponding $y_j \in \SA_{\twoD,j}$,
with~\eqref{eq:bd1} stating that for each $\eta \in (0, w \wedge z)$ there exists $M_0 = M_0(\eta)$ such that for all $M \geq M_0$ and $n \ge n_0(\eta,M)$
\begin{equation}
\label{prop::2d_excursion_count}
\Big| \frac{\log \p[ \ncylE_{y_k,k,j,w}(s) \le (z-\eta)^2 \cylE^\star(s)]}{\log n} + \frac{2 s (w - z)_+^2}{\rho_j-\rho_k} \Big| \leq \eta \,.
\end{equation}

In Lemma~\ref{lem::excursion_count_approx} we
stochastically dominate $\ncylE_{y_k,k,j,w}(s)$ from
above and below by comparable variables of a
much simpler form and thereby establish
\eqref{prop::2d_excursion_count} upon
studying in Lemma \ref{lem::cramer_excursion}
the tail behavior of the latter variables. 
Specifically, fixing $0 \le k < j \le L$, set for
each $n \in \N$,
\[ p_{k \to j}(n) := \frac{\log R_k - \log R_k'}{\log R_j - \log R_k'} \quad\text{and}\quad p_{j \to k}(n) := \frac{\log R_j - \log R_j'}{\log R_j - \log R'_k}\,.\]
As explained in \cite[Chapter 1]{lawler}, the hitting probabilities 
for \abbr{srw} $X$ within large size cylindrical annulus, have the same 
asymptotic as such probabilities for the corresponding 2D Brownian motion. 
In particular, $p_{k \to j}(n)$ (resp.\ $p_{j \to k}(n)$) approximates the probability that the \abbr{srw} $X$ starting from a point in $\partial \cyl(y_k,R_k)$ (resp.\ $\partial \cyl(y_j,R'_j)$) hits $\partial \cyl(y_j,R_j)$ before hitting $\partial \cyl(y_k,R_k')$ 
(resp.\ hits $\partial \cyl(y_k,R'_k)$ before hitting 
$\partial \cyl(y_j,R_j)$).
Moreover, it is easy to check that
\begin{equation}
\label{eqn::prob_asymp}
\lim_{M \to \infty} \lim_{n \to \infty} \frac{p_{k \to j}(n) \cylE^\star(s)}{\log n}  = \lim_{M \to \infty} \lim_{n \to \infty} \frac{p_{j \to k}(n) \cylE^\star(s)}{\log n} =  \frac{2 s}{\rho_j-\rho_k}.
\end{equation}
We next show that the variables 
$\ncylE_{y_k,k,j,w}(s)$ are stochastically related to
\begin{equation}
\label{eq:zps}
Z_{w,s}(p,p') := \sum_{i=1}^{w^2 \cylE^\star(s)} J_i(1+Y_i),
\end{equation}
where the i.i.d.\ Bernoulli($p$) variables $(J_i)$ are independent of 
the i.i.d.\ Geometric($p'$) variables $(Y_i)$, provided the
parameters $p \in (0,1)$ and $p' \in (0,1)$ are comparable to 
$p_{j \to k}(n)$ and $p_{k \to j}(n)$, respectively.
\begin{lem}
\label{lem::excursion_count_approx}
For every $c > 1$, $w >0$ and $L \ge j > k \ge 0$,
all $M \geq M_0(c,L)$ and $n \geq n_0(c,L,M)$,
if $p > c p_{j \to k}(n)$ and $p' < p_{k \to j}(n)/c$, then the law of $\ncylE_{y_k,k,j,w}(s)$ is stochastically dominated from above by
$Z_{w,s}(p ,p')$. Likewise,
if $p < p_{j \to k}(n)/c$ and $p' > c p_{k \to j}(n)$ then the law of $\ncylE_{y_k,k,j,w}(s)$ is stochastically dominated from below by
$Z_{w,s}(p ,p')$.
\end{lem}
\begin{proof}
For each $i$, let $\wt{J}_i$ denote the indicator of the event that the $i$th excursion $E_i$ of the \abbr{srw}
$X$ from $\partial \cyl(y_j,R_j')$ to $\partial \cyl(y_j,R_j)$
hits $\partial \cyl(y_k,R_k')$.  We also let 
$\wt{Y}_i$ denote the number of returns that the \abbr{srw} $X$ makes to $\cyl(y_k,R_k')$ from $\partial \cyl(y_k,R_k)$ before exiting $\cyl(y_j,R_j)$ during $E_i$. Then,
\[ 
\ncylE_{y_k,k,j,w}(s) = \sum_{i=1}^{w^2 \cylE^\star(s)} \wt{J}_i (1+\wt{Y}_i).\]
Let $\CF_{j}$
denote the $\sigma$-algebra generated by the entrance and exit points of all excursions $\{E_i\}$ and $\CF_{j,k}$ denote the $\sigma$-algebra generated by
$\CF_j$ as well as all entrance and exit points of the  excursions of $X$ from
$\partial \cyl(y_k,R'_k)$ to $\partial \cyl(y_k,R_k)$.
By \cite[Lemma~2.3]{DPRZ-late} in the limit $M \to \infty$ 
the probability of the occurrence of $\wt{J}_i$ given 
$\CF_j$ does not depend on the relevant starting and ending points. The same applies
for the probability that $\wt{Y}_i=\ell$ given $\wt{Y}_i \ge \ell$ and 
$\CF_{j,k}$.
Thus, in view of \cite[Exercise 1.6.8]{lawler}, we conclude that,
\begin{equation}
\label{eqn::pn1}
\liminf_{M \to \infty} \liminf_{n \to \infty}
\inf_{i} \Big\{\frac{\p[\wt{J}_i=1 \giv \CF_j]}{p_{j \to k}(n)} \Big\}
 =
\limsup_{M \to \infty} \limsup_{n \to \infty}
\sup_{i} \Big\{\frac{\p[\wt{J}_i=1 \giv \CF_j]}{p_{j \to k}(n)} \Big\}
= 1,
\end{equation}
\begin{align}
\label{eqn::pn2}
\liminf_{M \to \infty} \liminf_{n \to \infty} \inf_{i,\ell}
&\Big\{\frac{\p[\wt{Y}_i = \ell \giv \CF_{j,k},\ \wt{Y}_i \geq \ell]}
{p_{k \to j}(n)}\Big\} \notag \\
& =
\limsup_{M \to \infty} \limsup_{n \to \infty} \sup_{i,\ell}
\Big\{\frac{\p[\wt{Y}_i = \ell \giv \CF_{j,k},\ \wt{Y}_i \geq \ell]}
{p_{k \to j}(n)}
\Big\} = 1.
\end{align}
Combining~\eqref{eqn::pn1} and~\eqref{eqn::pn2} yields the desired result because the excursions $\{E_i\}$ are conditionally independent given $\CF_j$.
\end{proof}

By Lemma~\ref{lem::excursion_count_approx},
it suffices to prove the bounds of
\eqref{prop::2d_excursion_count} for $Z_{w,s}(p_n,p_n')$
in place of $\ncylE_{y_k,k,j,w}(s)$, provided that 
both $p_n/p_{j \to k}(n) \to 1$ and
$p_n'/p_{k \to j}(n) \to 1$. Further, 
in view of~\eqref{eqn::prob_asymp}, when doing so we may consider w.l.o.g.
$p_n'=\kappa p_n$, $\kappa \in (0,\infty)$, taking $n \to \infty$
followed by $\kappa \to 1$. To this end, set
\[ \Lambda_{p,p'} (\theta) := \log \E[ e^{- \theta
J_1(1+Y_1)}] = \log\left( 1 - p + \frac{p p'}{e^{\theta} - 1 + p'} \right) \quad\text{for}\quad \theta \geq 0,\]
and for each $0 \le z \le w \le 1$, let
\[
I_{p,p'}(z,w) := \frac{1}{p} \inf_{\theta \ge 0} \Big\{ z^2 \theta + w^2 \Lambda_{p,p'} (\theta) \Big\},
\]
whose asymptotic as $p'=\kappa p$, $p \to 0$ shall describe 
the tail behavior of $Z_{w,s}(p_n,p_n')$ which is relevant here.
\begin{lem}
\label{lem::i_lambda_properties}
Fix $\kappa \in (0,\infty)$.  Then, we have that
for $w \ge \sqrt{\kappa} z > 0$,
\begin{equation}
\label{eqn::I-val}
I_\kappa (z,w) := \lim_{p \to 0} I_{p,p\kappa} (z,w) = \inf_{v \ge 0} \left( v z^2  - \frac{v w^2}{\kappa+v}  \right) =
- (w - \sqrt{\kappa} z)^2.
\end{equation}
Let $\theta_p \in [0,\infty)$ be the unique value so that $\Lambda_{p,\kappa p}'(\theta_p) = -(z/w)^2$.  Then,
\begin{align}
\label{eqn::theta_p_limit}
\lim_{p \to 0} &\frac{\theta_p}{p} =
\sqrt{\kappa} \frac{w}{z} - \kappa := v_{\star} \ge 0,\\
\label{eqn::mgf_2nd_deriv_limit}
 \lim_{p \to 0} &p^2 \Lambda_{p,\kappa p}''(\theta_p) = 0.
\end{align}
\end{lem}
\begin{proof}
We begin by making the substitution $\theta := \log(1+p v)$ for $v \geq 0$, and setting $f_p(v):=p^{-1} \log (1+pv)$
rewrite $I_{p,\kappa p}(z,w)$ as
\begin{equation}
\label{eqn::I_general_def}
I_{p,\kappa p}(z,w) = \inf_{v \geq 0}
\Big\{  z^2 f_p(v) + w^2 f_p(\frac{-v}{\kappa+v}) \Big\}.
\end{equation}
Since $f_p(v) \uparrow \infty$ as $v \to \infty$,
the infimum in~\eqref{eqn::I_general_def} is attained
at some finite $v_p$. Further, with $p \mapsto f_p(v)$
non-increasing,
there exists a universal finite constant $V$ such that
$v_p$ takes its values in $[0,V]$ as $p \to 0$ and $\kappa$ fixed. This allows us to change the order of the limit in $p$ and the infimum
over $v$, yielding
\[
I_\kappa (z,w) = \inf_{v \geq 0} \lim_{p \to 0} \Big\{ z^2
f_p(v) + w^2 f_p(\frac{-v}{\kappa+v}) \Big\}.
\]
Since $f_p(v) \to v$ for $p \to 0$, the first assertion
of the lemma follows
upon verifying that the infimum in~\eqref{eqn::I-val}
is achieved at $v_{\star} \ge 0$.

As for confirming~\eqref{eqn::theta_p_limit}
and~\eqref{eqn::mgf_2nd_deriv_limit},
let
$F_p(v):=f_p(F_0(v))$ for $F_0(v)=-v/(\kappa+v)$, so
$\Lambda_{p,\kappa p}(\theta) = p F_p(v)$, under the substitution
$\theta = \log(1+pv)$. Differentiating both sides of this identity twice and rearranging, we find that
\begin{equation}\label{eq:theta-v}
p^2 \Lambda_{p,\kappa p}''(\theta) =
p(1+p v) \big( F_p''(v)(1+pv) + pF_p'(v) \big).
\end{equation}

Since the infimum in the definition of $I_{p,\kappa p}(z,w)$
is attained at $\theta_p$, necessarily $\theta_p = p f_p(v_p)$.
Thus, as $p \to 0$ we have that
$p^{-1} (e^{\theta_p}-1) = v_p \to v_{\star}$,
from which~\eqref{eqn::theta_p_limit} follows.
Further, $F_p'(v_p) \to F_0'(v_\star)$ and
$F_p''(v_p) \to F_0''(v_\star)$, yielding
\eqref{eqn::mgf_2nd_deriv_limit} in view of
\eqref{eq:theta-v}.
\end{proof}

As explained before, the required bounds 
\eqref{prop::2d_excursion_count}
are established by combining Lemma~\ref{lem::excursion_count_approx}
with our next lemma,
then taking $\kappa \to 1$ (we have the required
boundedness of $p_n \log n$ by 
\eqref{eq:typical-counts} and~\eqref{eqn::prob_asymp}).

\begin{lem}
\label{lem::cramer_excursion}
Fix $s \geq 1$, $\kappa \in (0,\infty)$ and $w \ge \sqrt{\kappa} z>0$.
If $p_n \log n$ are uniformly bounded above and uniformly bounded
away from zero, then
\begin{equation}
\label{eqn::cramer_excursion}
\lim_{n \to \infty} \frac{1}{p_n \cylE^\star(s)} \log \p[ Z_{w,s}(p_n,\kappa p_n) \leq z^2 \cylE^\star(s) ] = - (w - \sqrt{\kappa} z)_+^2.
\end{equation}
\end{lem}
\begin{proof}
Fix $s \ge 1$, $\kappa \in (0,\infty)$ and $w \ge \sqrt{\kappa} z >0$. Now, for any $p \in (0,1)$
we get by applying Chernoff's bound, then optimizing over $\theta \geq 0$ that
\begin{equation}
\label{eqn::cramer_excursion_ubd}
\frac{1}{p \cylE^\star(s)} \log \p[ Z_{w,s}(p,\kappa p) \leq
z^2 \cylE^\star(s) ] \le I_{p, \kappa p}(z,w)\,.
\end{equation}
Thus, in view of~\eqref{eqn::I-val},
considering $p=p_n \to 0$ yields the
upper bound in~\eqref{eqn::cramer_excursion}.

For the lower bound we use a change of measure
analogous to the proof of the lower bound in
Cramer's theorem (see \cite[Theorem~2.2.3]{DZ-LD}).
Specifically, fixing $p \in (0,1)$ and $\delta > 0$
small (we eventually send $\delta \to 0$), set
$\theta = \theta_p \geq 0$ be the unique value
such that $\Lambda_{p,\kappa p}'(\theta_p) = -(z-\delta)^2/w^2$ and probability
measure $\p_\theta$ given by
\[
\frac{d\p_\theta}{d\p} =
\exp\Big( -\theta Z_{w,s}(p,\kappa p)
- w^2 \cylE^\star(s) \Lambda_{p,\kappa p}(\theta) \Big)\,.
\]
Considering event
$A_{p,\kappa p} = \{ (\cylE^\star(s))^{-1}
Z_{w,s}(p,\kappa p) \in [(z-2\delta)^2, z^2] \}$,
we clearly have then
\begin{align}
\p[A_{p,\kappa p}]
\geq
\p_\theta[A_{p,\kappa p}] \exp\big( w^2 \cylE^\star(s) \Lambda_{p,\kappa p}(\theta) + \theta (z-2\delta)^2 \cylE^\star(s) \big). \label{eqn::prob_lower_bd}
\end{align}
Adding and subtracting
$\theta (z-\delta)^2 \cylE^\star(s)$
in the exponent on the \abbr{RHS} of
\eqref{eqn::prob_lower_bd},
then setting there $\theta=\theta_p$, we see that
$\p[A_{p,\kappa p}]$
is further bounded below by
\begin{align*}
\p_{\theta_p} [ A_{p,p'} ] \exp\big( p \cylE^\star(s)
I_{p, \kappa p}(z-\delta,w) - \eta \cylE^\star(s) \theta_p\big),
\end{align*}
where $\eta:=(z-\delta)^2 - (z-2\delta)^2$.
We now complete the proof by taking $p = p_n$ (we will suppress the subscript $n$). Indeed,
note that under $\p_\theta$ the variables $J_i(1+Y_i)$ are
i.i.d.\ each having mean $(z-\delta)^2/w^2$ and variance
$\Lambda_{p,\kappa p}''(\theta)$.  Further, $p_n^2 \cylE^\star(s)$ is bounded away from zero,
so by~\eqref{eqn::mgf_2nd_deriv_limit} we see that
$\var_{\theta_p}\big(\cylE^\star(s)^{-1} Z_{w,s}(p,\kappa p)\big)
\to 0$ as $n \to \infty$, while
$\E_{\theta_p} [\cylE^\star(s)^{-1} Z_{w,s}(p,\kappa p)] = -w^2 \Lambda'_{p,\kappa p}(\theta_p)=(z-\delta)^2$. Consequently,
\[ \lim_{n \to \infty} \frac{1}{p \cylE^\star(s)} \log \p_{\theta_p} [ A_{p,\kappa p}] = 0.\]
Hence, by~\eqref{eqn::I-val}
and
\eqref{eqn::theta_p_limit} we have that
\[ \liminf_{n \to \infty} \frac{1}{p \cylE^\star(s)} \log \p[Z_{w,s}(p,\kappa p) \leq z^2 \cylE^\star(s)] \geq -(w - \sqrt{\kappa} (z-\delta))_+^2 - 2 \eta .\]
The stated lower bound follows
by considering $\delta \to 0$ (so
$\eta \to 0$ as well).
\end{proof}

\section{Lower bound on mixing time: 
effective clustering in $\CU(s\tcovp)$}\label{sec-lbd}

Let $\Q_{s'}$ denote the 
law of the lamps configuration of $X^\diamond$ at time $s' \tcovp$,
starting  from all lamps off (and walker at the point 
$0 \in \ttorus$), with $\Q_\infty$ the uniform
law over the set of $2^{|V_n|}$ possible lamp configurations.
We claim that $\| \Q_{s'} - \Q_\infty \|_{\TV} \to 1$ 
when $n \to \infty$, 
for fixed $s'=(1-\epsilon) s$, any $s<\Psi(\phi)$ and $\epsilon>0$.
Obviously, then $\tmix \ge s' \tcovp$ for such $s'$, which 
in view of the upper bound on $\tmix$ we proved in
Section~\ref{sec:ubd-mix}, establishes the stated 
cut-off and thereby proves Theorem \ref{thm:cut-off}.
 
To prove this claim, 
fix $\epsilon>0$ and $s <\Psi(\phi)$, noting that
in view of \eqref{eq:alp-rho} and
the variational formulation~\eqref{eq:t-lbd}
of $\Psi(\phi)$, there exist $\rho$ 
and $w > z > (1+w/\rho)\delta$, all in $(0,1]$, 
such that for small enough $\delta>0$,
\begin{equation}\label{eq:lbd-par}
b_{\rho}(w-\delta) \ge 2\delta \quad \text{and} \quad
\alpha(z+ 3 \delta) + \lambda (\rho-\delta) \le \rho - 5\delta \,,
\end{equation}
where further, by \eqref{eq:la-under-2} and the assumed range of $z$,
\begin{equation}\label{eq:h-val}
\lambda:=2s \frac{(w-z+\delta)^2}{(\rho-\delta)^2} < 2 
\quad \text{and} \quad 
A:= \frac{(z-\delta)\rho-w\delta}{w-z+\delta} > 0 \,.
\end{equation}

Using hereafter these parameters, we considerably shorten our proof by 
taking advantage of the results of \cite{DPRZ-late} and 
\cite{DPR-quarter} (which we apply here 
for the 2D projection of the \abbr{srw} on $\ttorus$).
For this purpose, we change our cylinders radii
somewhat and consider throughout this section 
\[
R_k=R'_{k+1}=(k!)^3, \quad k=1,\ldots,m
\]
with $m \in \N$ such that for some $\bar{\gamma} \in [b+12,b+16]$
\[
n= K_m: = m^{\bar{\gamma}} R_m 
\]
(and $b \ge 10$ a universal constant from
\cite[Lemma~4.2]{DPRZ-late}).
Next, let
$\CZ_m$ denote a maximal set of $4R_{\rho m+4}$-separated 
points on the $2$D base of $\ttorus$ excluding those within 
distance $R_m$ of the starting position $0$ of the
2D projected \abbr{srw}, such that $(0, 2R_m)\in \CZ_m$ (so
$\CZ_m$ is precisely the set considered
in \cite[Equation~(10.3)]{DPRZ-late}, taking there 
$\beta = \rho$ and $K_m = n$). 
Further, set
$$
{\CZ}_m' := {\CZ}_m \bigcap \bigcup_{v_i} \cyl(v_i,R_{m-2}),
$$
for a collection $\{v_i\}$ that forms a maximal $4 R_{m}$-separated set
on the 2D base of $\ttorus$. Next, for any
$v \in \CZ_m'$, let $\kappa_n(v)$ count the vertices 
of $\cyl(v,R_{\rho m - 2}) \subset \ttorus$, and 
$D^v$ denote the difference of number of ``off-lamps'' 
minus ``on-lamps'' among these $\kappa_n(v)$ vertices.
Considering the statistics
$$
U_n = \max_{v \in {\mathcal Z}_m'} \; \{D^v\},
$$
it suffices to show that as $n \to \infty$,
\begin{align}\label{eq:tv-lbd}
\Q_\infty[U_n \ge n^{\rho+\delta}] \to 0
\quad \text{and} \quad
\Q_{s'}[U_n < n^{\rho+\delta}] \to 0 \,.
\end{align}
We proceed with the proof of~\eqref{eq:tv-lbd}, establishing 
in {\bf Step I} the easy part, namely its \abbr{lhs}. 
Introducing 
$n_m(2s):=6sm^2 \log m$ and
\begin{align}\label{dfn:hatU}
\wh{U}^v := \big|\{ x \in \cyl(v,R_{\rho m-2}) : \;\;
x & \hbox{ unvisited in first } n_m(2s)  
\hbox{ excursions by} \nonumber \\ 
& \hbox{the \abbr{srw} from } \partial \cyl(v,R_m') 
\hbox{ to } \partial \cyl(v,R_m) \ \big| \,,
\end{align}
we reduce in {\bf Step II}
the \abbr{rhs} of~\eqref{eq:tv-lbd} 
to having \abbrs{whp} some $v \in \CZ_m'$ with large enough  
$\wh{U}^v$ (see \eqref{eq:shift}). We now need the following additional notations.
\begin{defn}\label{def:barU}
For a maximal set ${\CZ}_{\delta m}(v)$ of
$4 R_{\delta m}$-separated points in the $\twoD$ 
projection of $\cyl(v,R_{\rho m -2})$ on the base of $\ttorus$, let:
\newline
(a) $W^v$ count
points in $\CZ_{\delta m}(v)$ for
whose $R_{\delta m}$-sized cylindrical annulus 
the \abbr{srw} completed at most $z^2 n_m(2s)$ excursions
during its first $w^2 n_m(2s)$ excursions from 
$\partial \cyl(v,R'_{\rho m})$ to $\partial \cyl(v,R_{\rho m})$.
\newline
(b) ${\bar U}^v \le W^v$
count those $y$ from (a), for which in addition
$x=(y,0)$ is not visited during the first $z^2 n_m(2s)$ excursions
from $\partial \cyl(x,R_{\delta m}')$ to $\partial \cyl(x,R_{\delta m})$.
\end{defn}
\noindent
{\bf Step III} shows that \abbrs{whp}
$\wh{U}^{v_\star} \ge {\bar U}^{v_\star}$ for some $v_\star \in \CZ_m'$. 
Indeed, it clearly suffices to have at most $w^2 n_m(2s)$ of the
$R_{\rho m}$-excursions of $v_\star$ within the first $n_m(2s)$ 
of its $R_m$-excursions. This applies to
pre-qualified points from \cite[Section 10]{DPRZ-late}, so 
we complete this step by showing that \abbrs{whp} 
the relevant count $\Wpq(m)$ of pre-qualified points 
(see \eqref{def:WPQ}), is positive.
{\bf Step~IV} then converts the conditional statement of bounding below ${\bar U}^{v_\star}$ (for the random $v_\star$), into such a statement for non-random $v$,
which we verify under the condition of $W^v$ large enough
(see \abbr{rhs} of \eqref{eq:A-holds}).
We complete the proof of the latter (see \abbr{lhs} of \eqref{eq:A-holds}),
by applying in {\bf Step~V} the concept of pre-sluggish 
points from \cite[Section 6]{DPR-quarter}.

\medskip
\noindent
{\bf Step I}. Note that under $\Q_\infty$ the
variables $\{D^v, v \in {\CZ}_{m}'\}$ are
mutually independent, with $D^v$ having the law
of the sum of $\kappa_n(v)$ i.i.d.\ symmetric
$\{\pm 1\}$-valued variables $\{I^v_j\}$.
Further, $\sup_v \kappa_n(v) \le C h n^{2\rho}$
and $|{\mathcal Z}_m'| \le C n^{2(1-\rho)}$
for some $C$ finite and all $n$.
Recall that $\E[e^{\zeta I^v_j}] \le e^{\zeta^2/2}$
for all $\zeta$, hence by the union bound over at
most $C h n^2$ values of $v \in {\mathcal Z}_{m'}$
and the uniform tail bound
\begin{align}\label{eq:unif-tail-bd}
\sup_{r \le C h n^{2\rho}}
\p\Big[ \sum_{j=1}^r I^v_j \ge n^{\rho+\delta} \Big]
\le e^{-n^{\delta} + C h/2} ,
\end{align}
we conclude that the \abbr{lhs} of~\eqref{eq:tv-lbd}
holds for any $\delta>0$. 

\medskip
\noindent
{\bf Step II}. 
Turning to the \abbr{rhs} of~\eqref{eq:tv-lbd}, 
let $N\cyl_{v,m} (s')$ count the $R_m$-excursions 
for cylindrical annuli centered at $v$ on the 2D base of $\ttorus$, 
made by the \abbr{srw} on $\ttorus$ up to time $s' \tcovp$.
Note that $\log n =(3+o(1)) m \log m$ and
for $R/R'=R_m/R_m'=m^3$ the value of $\cylE^\star(s)$ of~\eqref{eq:typical-counts} 
is within $1+o(1)$ (as $n \to \infty$), of $n_m (2s)= 6s m^2 \log m$ (from \cite{DPRZ-late}). 
Hence, analogous to part (a) of Definition \ref{dfn:Gz} we have that
\begin{equation}\label{eq:As-likely}
\lim_{n \to \infty} \p\big(\max_v \{ N \cyl_{v,m}(s') \} > n_m(2s) \big) = 0,
\end{equation}
where the maximum is over all $n^2$ vertices $v$ on the $2$D base of $\ttorus$.
Indeed, combining the tail bound \cite[Equation~(3.18)]{DPRZ-late} for
the aggregate number of steps during the first $n_m(2s)$ such 
$R_m$-excursions for fixed $v$, with standard exponential 
tail bounds on the number of actual steps taken by our 
$\frac{2}{3}$-lazy projected $2$D \abbr{srw}, we thus deduce 
that $n^{2} \p(N \cyl_{v,m}(s') > n_m(2s) ) \to 0$ and the 
union bound over $v$ results with \eqref{eq:As-likely}.

We now show that the \abbr{rhs} of~\eqref{eq:tv-lbd} holds as soon as
\begin{equation}\label{eq:shift}
\lim_{n \to \infty} \p[ \max_{v \in {\CZ}_m'}
\{ \wh{U}^v \} < 2 n^{\rho+\delta}] = 0 \,,
\end{equation}
for $\wh{U}^v$ of \eqref{dfn:hatU}.
Indeed, $\Q_{s'}[U_n < n^{\rho+\delta}]$ is bounded
by the sum of the probabilities considered in \eqref{eq:As-likely}
and in~\eqref{eq:shift}, and
\begin{align}\label{eq:ubd-Gv}
\sum_{v \in {\mathcal Z}_m'}
\Q_{s'} \Big[
\sum_{j \notin \widehat{U}^v} I^v_j \le -
n^{\rho+\delta} \Big] \,.
\end{align}
Further, conditional on the whole path of the
\abbr{srw} on $\ttorus$
the variables $\{I^v_j, j \notin \widehat{U}^v\}$
retain under $\Q_{s'}$ their
symmetric i.i.d.\ $\pm 1$-valued law, so the
sum of probabilities considered in (\ref{eq:ubd-Gv}) is small 
by the uniform tail bound of~\eqref{eq:unif-tail-bd}.

\medskip
\noindent
{\bf Step III}. Proceeding to prove~\eqref{eq:shift},
let $N_{m,k}^v$ denote the number of \abbr{srw} excursions 
from $\partial \cyl(v,R'_k)$ to $\partial \cyl(v,R_k)$, 
during the first $n_m(2s)$ excursions it made 
from $\partial \cyl(v,R'_m)$ to $\partial \cyl(v,R_m)$. 
We rely on \cite[Section 10]{DPRZ-late} to prove the 
existence \abbrs{whp} (as $m \to \infty$), of
$v \in \CZ_{m}'$ such that for $w$ of \eqref{eq:lbd-par} 
\begin{equation}\label{eq:pre-qual-conc}
N_{m,\rho m}^v < w^2 n_m(2s) \,.
\end{equation}
Indeed, we consider the choice of parameters 
$a=2s$, $\beta =\rho$ and $\gamma = (w-\delta)/\rho$,
in \cite[Section 10]{DPRZ-late} 
and call $v \in \CZ'_m$ an \emph{$(m,\beta)$-pre-qualified point} 
if $N_{m,k}^v \in [\widehat{n}_k - k,\widehat{n}_k+k]$
for all $\beta m \le k \le m-1$ and the value of 
$\widehat{n}_k$ given in \cite[Equation~(10.2)]{DPRZ-late}.
Since our 
choices of $a$, $\beta$ and $\gamma$ result with 
$\wh{n}_{\beta m}=(w-\delta)^2 n_m(2s)(1+o(1)) \gg m$,
we deduce that for some universal $m_0$ and 
all $m \ge m_0$, every $(m,\beta)$-pre-qualified point 
satisfies \eqref{eq:pre-qual-conc}. Further,
in view of \eqref{eq:b-def} and
the \abbr{lhs} of \eqref{eq:lbd-par},
the value of $a^\star$ in \cite[Section 10]{DPRZ-late}
(given the preceding choices of $a$, $\beta$, and $\gamma$), 
is such that 
$(1-\beta)(2-a^\star)=2 b_\rho(w-\delta) \ge 4\delta$.
Thus, letting 
\begin{equation}\label{def:WPQ}
\Wpq (m) := 
|\{v \in \CZ_m': v \mbox{ is } 
(m,\beta)\mbox{-
pre-qualified} \}| \,,
\end{equation}
it suffices to show that 

$$
\lim_{m \to \infty} \P\big( \Wpq (m) \ge K_m^{(1-\beta)(2-a^\star)-\delta} 
\big) = 1 \,,
$$ 
which we get by adapting the proof of \cite[Equation~(10.3)]{DPRZ-late}, 
in replacing the $(m,\beta)$-qualified points in $\CZ_m$ dealt with there, 
by the $(m,\beta)$-pre-qualified points in $\CZ'_m$ considered here. 
To this end, recall that \cite[Equation~(10.3)]{DPRZ-late} is derived 
by showing that:
\newline 
(a) The mean number of such points far exceeds 
$K_m^{(1-\beta)(2-a^\star)-\delta}$.
\newline
(b) Its variance is negligible relative to the square of its mean.
\newline
We further note that the $(m,\beta)$-qualified points of 
\cite[Section 10]{DPRZ-late} are essentially
our $(m,\beta)$-pre-qualified points for which also 
the event $\wh{\bf A}^v_{N^v_{m,\beta m}}$ 
as in the proof of \cite[Lemma 10.1]{DPRZ-late}, occurs.  
In \cite{DPRZ-late} one takes $2s<2$ for which 
the latter event is shown to occur \abbrs{whp}
(see \cite[Equation~(10.8)]{DPRZ-late}).
The probability that $v$ is $(m,\beta)$-qualified, as computed
in \cite[Equation~(10.4)]{DPRZ-late}, is thus within $(1+o(1))$
of the probability that $v$ is $(m,\beta)$-pre-qualified, and it is further 
easy to check that in the pre-qualified case the same formula
applies also when $2 s \ge 2$.

Hence, the same argument as in \cite{DPRZ-late} 
establishes (a) here as well. The key to (b) is
the bound of \cite[Equation~(10.7)]{DPRZ-late} which builds 
on the correlation upper bounds \cite[Equations~(10.5),(10.6)]{DPRZ-late}.
The latter have already been derived there for $(m,\beta)$-pre-qualified points 
and all $s>0$. Thus, \cite[Equation~(10.7)]{DPRZ-late} applies here as well,
apart for a minor difficulty due 
to the fact that we consider only points from the subset $\CZ_m'$ 
of $\CZ_m$. However, $\inf_m \{ m^{12} |\CZ_m'|/|\CZ_m| \}$ 
is positive, and we have already increased by $m^{12}$ 
the value of $K_m=n$, which as seen by following the derivation of
\cite[Equation~(10.7)]{DPRZ-late}, well compensates this effect.

\medskip
\noindent
{\bf Step IV}. Ordering the points of ${\CZ}_m'$ 
in some non-random fashion, we let $v^\star$ denote the first 
$v \in {\CZ}_m'$ satisfying \eqref{eq:pre-qual-conc} (which 
by {\bf Step III} exists \abbrs{whp}).  By definition the points in ${\CZ}_m'$ are
$4 R_{\rho m+4}$-separated and the 
$R_m$-sized cylindrical annulus around each is
of distance $R_{m-1} \ge R_{\rho m}$ from any
(other) point of ${\CZ}_m'$. Consequently, 
$v^\star$ is measurable 
on the $\sigma$-algebra $\CF$ generated 
by the \abbr{srw} path excluding the interior parts of 
excursions between $\partial \cyl(v,R_{\rho m-1})$ 
and $\partial \cyl(v,R_{\rho m})$, for all $v \in \CZ_m'$ 
(namely, each such part has been replaced by its entrance and exit points).
Recalling Definition \ref{def:barU} of the counts 
$\bar{U}^v \le W^v$ of points in $\CZ_{\delta m}(v)$,
we thus get \eqref{eq:shift} by showing that
\begin{equation}
\label{eq:A-vs-holds}
\lim_{n \to \infty}\p({\bar U}^{v_\star} \ge 2 n^{\rho+\delta} |\CF)=1 \,.
\end{equation}
Further, applying \cite[Lemma~2.4]{DPRZ-late} for $r=R_{\rho m-2}$,
$R=R_{\rho m-1}$, $R'=R_{\rho m}$ and the event
$\{ {\bar U}^v \ge 2n^{\rho+\delta}\}$ 
which is measurable on the 
$\sigma$-algebra $\CH^v(\ell)$ 
of the interior parts of first $\ell=w^2 n_m(2s)$ excursions 
for $R_{\rho m}$-sized
cylindrical annulus around $v$, we get the 
conditional result \eqref{eq:A-vs-holds}, once we show that 
for $\theta:=(2-\lambda)(\rho-\delta)-\delta$
and any non-random $v \in \CZ_m'$, as $n \to \infty$, 
\begin{equation}
\label{eq:A-holds}
\p(W^v \ge n^{\theta}) \to 1  \quad 
\text{and} \quad 
\p({\bar U}^{v}  \ge 2 n^{\rho+\delta} \,|\, W^v \ge n^{\theta}) \to 1 \,.
\end{equation}
Proceeding to establish the \abbr{rhs}, let $q_n$ be the minimal value over all 
possible excursion end points and the choice of 
$x \in \ttorus \setminus \cyl(0,R_{\delta m})$,
of the conditional probability that $x$ is not visited 
during the first $z^2 n_m(2s)$ of the \abbr{srw} 
excursions from $\partial \cyl(x,R'_{\delta m})$ 
to $\partial \cyl(x,R_{\delta m})$. Since
points in ${\CZ}_{\delta m}(v) \subset \cyl(v,R'_{\rho m}-R_{\delta m})$ 
are $4 R_{\delta m}$-separated, the variable $W^v$ 
is measurable on the $\sigma$-algebra $\CF^v$ generated 
by the \abbr{srw} path excluding the interior part of the 
excursions between $\partial \cyl(x,R'_{\delta m})$ 
and $\partial \cyl(x,R_{\delta m})$, for all $x=(y,0)$ and
$y \in \CZ_{\delta m}(v)$ 
(namely, each such part has been replaced by its entrance and exit points).
Thus, conditionally on $W^v \ge n^\theta$, the variable 
${\bar U}^v$ stochastically dominates the 
Binomial$(n^{\theta},q_n)$ law. From~\eqref{eq:lbd-par} 
and our choice of $\theta$ we have that
$$
\theta - \alpha(z+3 \delta) \ge \rho + 2 \delta \,,
$$
so by the \abbr{clt} for Binomial random
variables, we get the \abbr{RHS} of~\eqref{eq:A-holds} upon
proving that as $n \to \infty$,
\begin{equation}\label{eq:qn-lbd-by-alpha}
n^{\alpha(z+3\delta)} q_n \to \infty \,.
\end{equation}
In view 
of the \abbr{lbd} of Proposition \ref{prop::prob_hit_point},
we have \eqref{eq:qn-lbd-by-alpha} upon showing that for 
any $M$ large enough, the probability of having at least 
$(z + 2 \delta)^2 \cylE^\star(s)$ excursions from 
$\partial \cyl(x,M h)$ to $\partial \cyl(x,M^2 h)$
during the first $z^2 n_m(2s)$ of the corresponding 
$R_{\delta m}$-excursions, is bounded away from one,
uniformly in $x$, $m \to \infty$, and the possible
excursion end points. Further, the stochastic comparisons
of Lemma \ref{lem::excursion_count_approx} extend to 
our case where $R_0=M R_0'=M^2 h$ as before, but we 
replace $R_j=n^{j/L} R_0=M R_j'$ 
with $R_{\delta m}=(\delta m)!^3=(\delta m)^3 R'_{\delta m}$
and change $\cylE^\star(s)$ in \eqref{eq:zps} to $n_m(2s)$.
Since $n_m(2s)$ is 
within factor $1+o(1)$ of the value of 
$\cylE^\star(s)$ from \eqref{eq:typical-counts} that 
corresponds to $R'=R'_{\delta m}$ and
$R=R_{\delta m} =(\delta m)^3 R'$, 
the desired uniform bound on probabilities follows 
from the convergence $Z_{z,s}(p,p')/\E[Z_{z,s}(p,p')] \to 1$ as 
$m \to \infty$ followed by $M \to \infty$ (while both 
$p=3 \log m/\log R_{\delta m}$ and $p'=\log M/\log R_{\delta m}$ decay to zero).

\medskip
\noindent
{\bf Step V}. We set $\widehat{R}:=R_{\rho m}+R_{\rho m-2}$,
$\widehat{\rho}:=R_{\rho m-1}-R_{\rho m -2}$ and 
$\wh{n}_k(\lambda):= 3 \lambda (k+Am)^2 \log m$, $k=1,2,\ldots$
for $\lambda<2$ and $A>0$ of \eqref{eq:h-val}. Following the
proof of \cite[Lemma 6.1]{DPR-quarter} we call  
$y \in \CZ_{\delta m}(v)$ \emph{$(m,\rho)$-pre-sluggish} 
if for the universal constant $b \ge 4$ found there,
and all $\delta m \le k \le \rho m - b$ the \abbr{srw}
completed within $\pm k$ of $\wh{n}_k(\lambda)$ 
excursions from $\partial \cyl(y,R'_k)$ to $\partial \cyl(y,R_k)$
during its first 
$\wh{n}_{\rho m}(\lambda)$ excursions from
$\partial \cyl(y,\widehat{\rho})$ to $\partial \cyl(y,\widehat{R})$.
It is easy to check that $\wh{n}_{\rho m}(\lambda) = w^2 n_m(2s)$ and 
$\wh{n}_{\delta m}(\lambda) = (z-\delta)^2 n_m(2s) \le z^2 n_m(2s) - \delta m$
(these analogs of \cite[(6.4) and (6.5)]{DPR-quarter} are 
behind our choice of $A$ and $\lambda$ in \eqref{eq:h-val}).
Further, if $y \in \CZ_{\delta m}(v)$ then 
$$
\cyl(y,\wh{\rho}) \subseteq \cyl(v,R'_{\rho m}) \subset
\cyl(v,R_{\rho m}) \subseteq \cyl(y,\wh{R}) \,.
$$
Hence, 
$W^v$ exceeds the number $\wh{W}^v$ of $(m,\rho)$-pre-sluggish $y \in \CZ_{\delta m}(v)$. 
The latter points match the definition made in 
\cite[proof of Lemma 6.1]{DPR-quarter}, upon taking there the 
parameters $\gamma:=\rho$, $\beta=w$ and $\eta:=\delta$.
Utilizing \cite[Lemma~6.2]{DPR-quarter} it is shown
in the course of proving \cite[Lemma~6.1]{DPR-quarter}
that $\wh{W}^v$ concentrates 
\abbrs{whp} around its mean value, which 
for our choice of parameters turns out to be
$R_m^{\theta+\delta-o_m(1)}$
(see \cite[Equations~(6.6) and~(6.7)]{DPR-quarter}).
The values of $\lambda$, $\beta$ and $\gamma$ 
we have here are outside the range considered in
\cite[Lemmas~6.1 and~6.2]{DPR-quarter}, but this 
restriction in \cite{DPR-quarter} is only 
relevant for the \emph{extra requirement} made
in \cite[Equation~(6.10)]{DPR-quarter} that any 
$(m,\gamma)$-pre-sluggish point should be
\abbrs{whp} also $(m,\gamma)$-sluggish. We completely 
abandoned this requirement, so the proof of \cite{DPR-quarter} 
easily extends to yield the \abbr{lhs} 
of \eqref{eq:A-holds}.

\end{document}